\documentclass[12pt,english]{amsart}         
\usepackage[utf8]{inputenc}  
\usepackage{fullpage}
\usepackage{fancybox}		   
\usepackage{bibentry}      
\usepackage{babelbib} 
\usepackage{url}
\usepackage{graphics,color,amsfonts}
\usepackage{amsfonts,amssymb,amsmath,amsthm,amscd,enumerate,framed}
\usepackage[british,french]{babel}
\usepackage{pdfsync}
\usepackage{fancyhdr}
\usepackage[babel=true]{csquotes}
\usepackage[dvips]{graphicx}      
\usepackage{lmodern}
\usepackage{ae,aecompl}
\usepackage{eso-pic}	% Nécessaire pour mettre des images en arrière plan
\usepackage{array}	
\usepackage{titlesec}
\usepackage{tikz}
\usepackage{tkz-graph}
\usetikzlibrary{matrix}

\titleformat{\section}[block]  {\fontsize{12}{15}\filcenter}{\thesection.}{1em}{\MakeUppercase}
\titleformat{\subsection}[runin]{\fontsize{12}{15}\bfseries}{\thesubsection.}{0.5em}{}
\titleformat{\subsubsection}[runin]{\fontsize{11}{15}\bfseries}{\thesubsubsection.}{0.5em}{}

\title[]{\bf \large ON THE MALLE CONJECTURE \\ AND THE GRUNWALD PROBLEM} 
\author[]{François Motte}     
\date{\today}

\makeindex

\newcommand{\Q}{\mathbb{Q}}
\newcommand{\Z}{\mathbb{Z}}
\newcommand{\N}{\mathbb{N}}

\newcommand{\R}{\mathbb{R}}
\newcommand{\F}{\mathbb{F}}
\newcommand{\PP}{\mathbb{P}^1}
\newcommand{\ip}{\mathfrak{p}}
\newcommand{\dd}{\rho}
\newcommand{\frobp}{\mathcal{F}_{\ip}}
\newcommand{\frob}{\mathcal{F}}

\makeatletter
\newcommand*{\house}[1]{%
  \mathord{%
    \mathpalette\@house{#1}%
  }%
}
\newcommand*{\@house}[2]{%
  \dimen@=\fontdimen8 %
      \ifx#1\scriptscriptstyle\scriptscriptfont
      \else\ifx#1\scriptstyle\scriptfont
      \else\textfont\fi\fi
      3 %
  \sbox0{%
    $#1%
      \vrule width\dimen@\relax
      \overline{%
        \kern2\dimen@
        \begingroup % to keep changes of \dimen@ in #2 local
          #2%
        \endgroup
        \kern2\dimen@
      }%
      \vrule width\dimen@\relax
      \mathsurround=1.5\dimen@ % outside margin
    $%
  }%
  \ht0=\dimexpr\ht0-\dimen@\relax
  \dp0=\dimexpr\dp0+2\dimen@\relax
  \vbox{%
    \kern\dimen@ % reinsert previously removed space
    \copy0 %
  }%
}

\theoremstyle{plain}

\newtheorem{theorem}{Theorem}[section]

\newtheorem{prop}[theorem]{Proposition}
\newtheorem{conj}[theorem]{Conjecture}

\newtheorem{lemme}[theorem]{Lemma}

\theoremstyle{definition}
\newtheorem{defn}[theorem]{Definition}
 
\newtheorem{rmq}[theorem]{Remark}

\keywords{Galois extensions, Inverse Galois theory, Malle conjecture, Grunwald problem, Algebraic covers, Specialization, Diophantine geometry}

\email{francois.motte@univ-lille.fr}

\address{Laboratoire Paul Painlev\'e, Math\'ematiques, Universit\'e Lille, 59655 Villeneuve d'Ascq Cedex, France}

\begin{document}

\subjclass[2010]{Primary 12E25, 12F12, 11R58, 11R44,  ; Secondary 14Gxx, 11Rxx, 12Fxx}

\selectlanguage{english}

\maketitle                        

\begin{abstract}
We contribute to the Malle conjecture on the number $N(K,G,y)$ of finite Galois extensions $E$ of some number field $K$ of finite group $G$ and of discriminant of norm $|N_{K/\Q}(d_E)|\leq y$. We prove the lower bound part of the conjecture for every group $G$ and every number field $K$ containing a certain number field $K_0$ depending on $G$ : $N(K,G,y) \geq y^{\alpha (G)}$ for $y\gg 1$ and some specific exponent $\alpha (G)$ depending on $G$. To achieve this goal, we start from a regular Galois extension $F/K(T)$ that we specialize. We prove a strong version of the Hilbert Irreducibility Theorem   which counts the number of specialized extensions $F_{t_0}/K$ and not only the specialization points $t_0$, and which provides some control of $|N_{K/\Q}(d_{F_{t_0}})|$. We can also prescribe the local behaviour of the specialized extensions at some primes. Consequently, we deduce new results on the local-global Grunwald problem, in particular for some non-solvable groups $G$.
\end{abstract}

\section{Introduction}

\subsection{The Malle conjecture.} In inverse Galois theory, the Malle conjecture is about the number, say $N(K,G,y)$, with $K$ a number field, $G$ a finite group and $y$ a positive number, of Galois extensions $E/K$ (in a fixed algebraic closure $\overline{K}$ of $K$), with Galois group $G$ and with discriminant $d_{E/K}$ of norm $N_K(d_{E/K})$ bounded by $y$. It is well-known that this number is finite. There is the following conjecture \cite{malledistrib1} :

\begin{conj}
There exists a constant $a(G)>0$, depending only on $G$, such that for every $\varepsilon>0$, we have
$$c_1 y^{a(G)} \leq N(K,G,y) < c_2 y^{a(G)+\varepsilon} \text{  for all } y\geq y_0$$
for positive constants $c_1$ (depending on $G$, $K$) and $c_2$, $y_0$ (depending on $G$, $K$, $\varepsilon$).
\end{conj}

This conjecture is open for most groups $G$ over any number field $K$. Malle proved it over $\Q$ for abelian groups \cite{malledistrib1}, \cite{malledistrib2}, Klüners and Malle proved it (also over $\Q$) for nilpotent groups $G$ using the Shafarevich result on the existence of at least one extension of group $G$ \cite{mallenilp}. Klüners also proved the lower bound part for dihedral groups of order $2p$ where $p$ is an odd prime \cite{klunersdiedral}. In this paper, we also are interested in the lower bound part which we more specifically define as the following statement.
\medskip

\begin{defn}
We say that the lower bound part of the Malle conjecture is true if there exists $\alpha(G)>0$, depending only on $G$, such that $$ N(K,G,y) \geq c_1 y^{\alpha(G)} \text{ for all } y\geq y_0$$
for some positive constants $c_1, y_0$ depending on $K$, $G$.
\end{defn}

Note that the lower bound part already provides a positive answer to the inverse Galois problem. Our contribution is the following result, valid for any finite group $G$.

\vskip 2mm
\noindent \textbf{Theorem A.}  \textit{ Let $G$ be a finite group. There exists a number field $K_0$ such that the lower bound part of the Malle conjecture is true over every number field $K$ containing $K_0$. More precisely, the field $K_0$ can be any number field for which $G$ is a regular Galois group over $K_0$}.
\vskip 2mm

Recall that $G$ is said to be a {\it regular} Galois group over $K_0$ if there is a Galois extension $F/K_0(T)$ of group $G$ that is $K_0$-{\it regular} (i.e. $F\cap \overline{K_0}= K_0$). Over $\Q$, regular Galois groups include $S_n$ $(n\geq 1)$ and many simple groups : $A_n$ $(n\geq 5)$, many $PSL_2(\F_p)$, the Monster group, etc. \smallskip

Theorem A generalizes a previous result of Pierre Dèbes \cite{debconjmalle} who proved the lower bound part in the special case $K=\Q$ and when $G$ is supposed to be a regular Galois group over $\Q$. \smallskip

Malle also predicts the value of the expected exponent $a(G)$ in his conjecture : \\ $(|G|(1-1/l))^{-1}$ where $l$ is the smallest prime divisor of $|G|$. Our exponent $\alpha(G)$ will also be given explicitly. It is smaller than $a(G)$. We explain why in \S 2.3. 
\smallskip

There is a more general conjecture for not necessarily Galois extensions which we discuss in \S \ref{sectionpresentation}.5. The same lower bound holds for this general form of the conjecture (theorem \ref{generalcase}).

\subsection{The Grunwald problem.}
Furthermore, our approach makes it possible to impose some local constraints to the extensions $E/K$ that we count. This relates to the famous Grunwald problem.
\smallskip

For every prime $\ip$ of $K$, the completion of $K$ is denoted by $K_{\ip}$. The completion of $E$ is then the compositum $EK_{\ip}$ (with respect to any prime $\mathcal{P}$ above $\ip$). 
The Grunwald problem asks whether the following is true :
\medskip

\centerline{\begin{minipage}{0.8\textwidth} \textit{$(*)$ ~  Given a finite set $S$ of primes of $K$ and some finite Galois extensions $\left( L^{\ip}/K_{\ip}\right)_{\ip\in S}$ with Galois group embedding into $G$, there is a Galois extension $E/K$ of group $G$ whose completion $EK_{\ip}/K_{\ip}$ at $\ip$ is $K_{\ip}$-isomorphic to $L^{\ip}/K_{\ip}$ for every $\ip \in S$.}
\end{minipage}}
\medskip

\noindent Such an extension $E/K$ is called {\it a solution} to the {\it Grunwald problem} $\left(G,\left( L^{\ip}/K_{\ip}\right)_{\ip\in S} \right).$
 
The case of abelian, and more generaly solvable groups, has been studied by Grunwald, Wang and Neukirch \cite{wan48} \cite{neukirch1} : in particular, the answer is positive if $G$ is of odd order. But in general, some Grunwald problems exist with no solution, for example, if $G$ is cyclic of order $8$ and if $S$ contains a prime of $K$ lying over $2$ \cite{wan48}. Nowadays, it is expected that there should be an exceptional finite set $S_{exc}$ of primes such that $(*)$ holds if the set $S$ of primes is disjoint from $S_{exc}$. Several works have been devoted to this weak form \cite{harari07}, \cite{debgazi}, \cite{dlan}. It was recently established for hypersolvable groups (e.g. nilpotent) over any number field \cite{HaWi2018}. For non solvable groups, a result due to Dèbes and Ghazi \cite{debgazi} shows that any Grunwald problem $(L^{\ip}/K_{\ip})_{\ip \in S}$ (with $S \cap S_{exc}=\emptyset$), additionally assumed to be unramified, always has a solution if $G$ is a regular Galois group over $K$.
\medskip

Our result on this topic needs the following terminology from \cite{debgazi}. If $(G,(L^{\ip}/K_{\ip})_{\ip\in S})$ is a Grunwald problem over $K$ and $M/K$ is a finite Galois extension, denote by $S_M$ the set of primes of $M$ obtained by choosing one prime $\mathcal{P}$ of $M$ over each $\ip\in S$. Denote by $(G,(L^{\ip}M_{\mathcal{P}}/M_{\mathcal{P}})_{\mathcal{P}\in S_M})$ the Grunwald problem over $M$ induced by the base changes $M_{\mathcal{P}}/K_{\ip}$, $\ip \in S$. The base changed problem does not depend on the choice of the primes $\mathcal{P}$.
\smallskip

Note next that if $M/K$ is totaly split at each $\ip \in S$ then $M_{\mathcal{P}}=K_{\ip}$ and $L^{\ip}M_{\mathcal{P}}/M_{\mathcal{P}}=L^{\ip}/K_{\ip}$ $(\ip \in S$). A solution $E/M$ of the base changed Grunwald problem $(G,(L^{\ip}/K_{\ip})_{\ip\in S_M})$ will be said to be a $M$-{\it solution} of the (original) Grunwald problem $(G,(L^{\ip}/K_{\ip})_{\ip\in S})$.

\vskip 2mm
\noindent \textbf{Theorem B.}  \textit{ Let $G$ be a finite group and $K$ be a number field. There exists a finite set $S_{exc}$ of primes of $K$ with the following property : if $(G,(L^{\ip}/K_{\ip})_{\ip\in S})$ is any unramified Grunwald problem over $K$ with $S\cap S_{exc}=\emptyset$, then there exist a finite Galois extension $M/K$, totally split at each $\ip \in S$ and an infinite set of $M$-solutions $E/M$ to the Grunwald problem $(G,(L^{\ip}/K_{\ip})_{\ip\in S})$.
\smallskip
Furthermore, one can take $M=K$ if $G$ is a regular Galois group over $K$.}
\vskip 2mm

In particular, for all non solvable groups known to be regular groups over $\Q$, any unramified Grunwald problem has infinitely many solutions over $\Q$.

\subsection{Diophantine results.}

Both theorem A and theorem B are special cases of a more general result, theorem AB, which will be stated in \S \ref{sectionpresentation}.$2$. We will start with a regular Galois extension $F/K(T)$ of group $G$ and will use the set of extensions $F_{t_0}/K$ obtained  from $F/K(T)$ by specializing $T$ to $ t_0 \in K$ \footnote{Definition of specialized extensions is recalled in \S 2.}. From the \textit{Hilbert Irreducibility Theorem}, these specialized extensions are still Galois of group $G$ for a large number of $t_0$. \medskip

The idea is, for theorem A, to count the number of these specialized extensions and for theorem B, to show that some local conditions can be prescribed to these extensions. We will follow a method developed in \cite{debgazi} and \cite{debconjmalle} over $\Q$ and which has an important diophantine part. A major tool will be an estimate of the number $N(F,B)$ of rational points on a curve of height bounded by a number $B$. This is a classical problem, for which Heath-Brown introduced a method in $2002$ \cite{heathbrown1}, which was refined for curves by Walkowiak \cite{walarticle}, both over $\Q$. In this context we will prove the following result, which extends Walkowiak's result to any number field and may be interesting for its own sake.
\medskip

Denote by $O_K$ the ring of integers of $K$. We will use the following height for number $x\in O_K$, sometimes called the {\it house} of $x$ :
 $$H(x)=\displaystyle \max (|x_1|,\cdots,|x_d|)$$
 where $x_1,\cdots,x_d$ are the $\Q$-conjugates of $x$ (see \S \ref{sectiondiophantine} for more on heights).

Consider a polynomial $F(X_1,X_2) \in O_K[X_1,X_2]$, irreducible in $K[X_1,X_2]$, monic in $X_2$. \smallskip

For $B>0$, define 
$$N(F,B)=\#\{(x_1,x_2) \in O_K^2 ~ : ~ F(x_1,x_2)=0, ~ H(x_1)\leq B ,~ H(x_2)\leq B \}.$$

\vskip 2mm
\noindent \textbf{Theorem C.}  \textit{ If $B$ is suitably large (depending on $K$), we have
$$N(F,B) \leq c \deg(F)^{8} \hskip 2pt (\log B)^{3} \hskip 2pt B^{[K:\Q] /\deg(F)}$$ where $c$ is a constant depending on $K$.}
\vskip 2mm

Having such an estimate available for any number field is  crucial for our applications. Proof of theorem C is inspired by Walkowiak's work over $\Q$ but has to deal with several new phenomena occuring on an arbitrary number field. \smallskip

Theorem C has the following consequence more in the spirit of \textit{Hilbert's Irreducibility Theorem} and which we will use for our results.

\vskip 2mm
\noindent \textbf{Corollary C.}  \textit{ Let $F(T,Y) \in O_K[T,Y]$ irreducible and monic in $Y$. There exist some positive constants $a_1,...,a_4$ depending on $K$ such that for all suitably large $B$, the number $N_T(F,B)$ of $t\in O_K$ with $H(t) \leq B$ and such that $F(t,Y)$ has a root in $K$ satisfies   
$$N_T(F,B)\leq a_1\deg(F)^{a_2} \hskip 2pt (\log H(F))^{a_3} \hskip 2pt B^{[K:\Q]/\deg_Y(F)} \hskip 2pt (\log B)^{a_4}$$
where $H(F)$ is the height of $F$. \footnote{see \S 3 for more details on the height $H(F)$.}}
\vskip 2mm
 
Recall that the total number of $t\in O_K^2$ with $H(t)\leq B$ is asymptotic to $B^{[K:\Q]}$ (up to some multiplicative constant and a $\log B$ factor) \cite{schanuel} \cite{countalgint}.
\smallskip

The paper is organized as follows. In \S 2.1, we present two key results about specialization : theorem \ref{SPECTHM} and theorem \ref{HilbertMalle}. They are intermediate between the pure diophantine statements (theorem C and corollary C) and our applications (theorem A and theorem B). How we use them to obtain the applications is done in 3 steps and explained in \S 2. In \S 2.2, theorem AB is stated. In \S 2.3, theorem AB is shown to imply theorems A and theorem B. In \S 2.4, theorem AB is proved  assuming theorems \ref{SPECTHM} and \ref{HilbertMalle}. \S \ref{sectiondiophantine} is dedicated to the proof of theorem C and corollary C. Finally, theorems \ref{SPECTHM} and \ref{HilbertMalle} are proved in \S 4.
 
 \smallskip
 
 The following figure summarizes the structure of our approach.
 
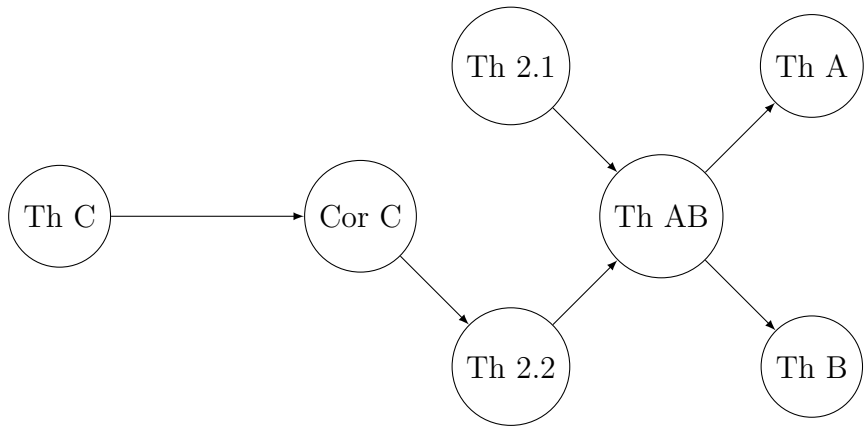
\begin{figure}[h]
\begin{center}
\begin{tikzpicture}[scale=2,every node/.style={circle}]
  \node[draw] (1) at (0,0)  {Th C};
    \node[draw] (2) at (2,0)  {Cor C};
        \node[draw] (3) at (3,1)  {Th 2.1};
            \node[draw] (4) at (3,-1)  {Th 2.2};
                    \node[draw] (5) at (4,0) {Th AB};
                    \node[draw] (6) at (5,1) {Th A};
                    \node[draw] (7) at (5,-1) {Th B};  

  \path

 (1) edge [->,>=latex] (2)
%  (2) edge [->] (3)
    (2) edge [->,>=latex] (4)
       (3) edge [->,>=latex] (5)
          (4) edge [->,>=latex] (5)
          (5) edge [->,>=latex] (6)
          (5) edge [->,>=latex] (7)
  ;
\end{tikzpicture}
\end{center}
%    \caption{\label{schema}}
\end{figure}
\newpage

\section{two specialization results and their applications} \label{sectionpresentation}

Both theorem \ref{SPECTHM} and theorem \ref{HilbertMalle} deal with specializations of a regular Galois extension $F/K(T)$ of group $G$. The first one is a version of \textit{Hilbert's Irreducibility Theorem} : it explicitely produces many $t_0$ such that the specialized extention $F_{t_0}/K$ is of group $G$. The second one shows that not so many of these specialized extensions $F_{t_0}/K$ can be isomorphic. \smallskip

We retain the following notation. Fix for the whole \S 2 a number field $K$ of degree $\dd = [K:\Q]$, a finite group $G$ and a $K$-regular Galois extension $F/K(T)$ of group $G$. Denote by $r$ the number of branch points of $F/K(T)$ (or equivalently of the associated cover $f: X\rightarrow \PP$) and the genus of $F$ (or of $X$) by $g$. For a prime $\ip$ of $K$, the prime number lying below $\ip$ is denoted by $p_{\ip}$ and we have $\ip \cap \Z = p_{\ip}\Z$. 
 \smallskip

Given a point $t_0 \in K$ (or $t= \infty$), the specialization of $F/K(T)$ at $t_0$ is the residue extension of the integral closure of the localized ring $K[T]_{\langle T-t_0 \rangle}$ in $F$ at an arbitrary prime above $\langle T-t_0 \rangle$. Denote it by $F_{t_0}/K$. If $P(T,Y) \in O_K[T,Y]$ is what we call an {\it affine model} of $F/K(T)$, i.e. the minimal polynomial of some primitive element of $F/K(T)$ integral over $K[T]$, then for all $t_0\in K$ not in the finite list of roots of the discriminant $\Delta_P(T)$ of $P$ with respect to $Y$, the specialization  $F_{t_0}/K$ is also the splitting field of $P(t_0,Y)\in K[Y]$.

\subsection{Statements of theorems \ref{SPECTHM} and theorem \ref{HilbertMalle}.}

Theorem \ref{SPECTHM} below gives a lower bound for the number of \og good {\fg} specialization points $t_0$ of bounded height.
\smallskip

Our statement also involves some local conditions that the specialized extensions should satisfy. Given a set $S$ of prime ideals of $O_K$, one defines a {\it Frobenius data} on $S$ as a collection $\mathcal{F}_S=(\mathcal{F}_{\ip})_{\ip \in S}$ of subsets $\frobp \subset G$, each $\frobp$ being a non-empty union of conjugacy classes of $G$. The set $S$ is said to be over the interval $[a,b]$ if $S$ is the set of all prime ideals over the prime numbers $p\in [a,b].$ Requiring that for each $\ip \in S$, the Frobenius ${\rm Frob}_{\ip}(F_{t_0}/K)$ lies in $\frobp$ will be the form of our local prescription to our specializations $F_{t_0}/K$. For example, if $\frobp=\{1\}$ for every $\ip \in S$, it is that $F_{t_0}/K$ should be totally split at each prime $\ip \in S$.
\smallskip

Choose 
\begin{itemize}
\item a prime number $p_{-1} \geq r^2g^2$ and such that every prime number $p$ which is ramified in $K/\Q$ is $\leq p_{-1}$ and
\item a prime number $p_0$ such that  the interval $]p_{-1},p_0[$ has at least as many prime numbers as there are conjugacy classes in $G$. \\ 
\end{itemize}
The primes $p_{-1}$ and $p_0$ depend on $K,r,g$ and $K,r,g,G$ respectively. For $B>0$, let $S_B$ be the set of primes of $K$ over the interval $[p_0,\log(B)/2]$.

\begin{theorem} [{\it Hilbert type}] \label{SPECTHM} There exists a number $c>0$ (depending on $F/K(T)$) such that if $B$ is suitably large (depending on $F/K(T)$), if $\frob_B=(\frobp)_{\ip \in S_B}$ is any Frobenius data on $S_B$, the number of $t_0 \in O_K$ of height $H(t_0)\leq B$ such that 
\begin{itemize}
\item the specialized extension $F_{t_0}/K$ is of group $G$,
\item ${\rm Frob}_{\ip} (F_{t_0}/K) \in \frobp$ for every $\ip \in S_B$.
\end{itemize}
 $$ \text{is at least }\displaystyle  \frac{B^{\dd}}{c^{\log B/\log \log B}}.$$ 
\end{theorem}

In the spirit of the Malle conjecture, we have to count not just the number of good specialization points $t_0$ but the number of different corresponding extensions $F_{t_0}/K$. Here enters the \textit{Hilbert-Malle type theorem} \ref{HilbertMalle} below. The special case $K=\Q$ was proved in \cite{debconjmalle}. We generalize it to arbitrary number fields.

\begin{theorem}[{\it Hilbert-Malle type}]\label{HilbertMalle} 
Let $B>0$ be a real number. Let $\mathcal{H} \subset O_K$ be a subset consisting of $t_0$ such that ${\rm Gal}(F_{t_0}/K)=G \text{ and } ~ H(t_0) \leq B$. Denote by $\mathcal{N}(B,\mathcal{H})$ the number of corresponding  specialized field extensions $F_{t_0}/K$ when $t_0  \in \mathcal{H}$.
There exist $E,\gamma \geq 0$ depending on $F/K(T)$ such that if $B$ is suitably large (depending on $F/K(T)$), we have $$\mathcal{N}(B,\mathcal{H}) \geq \frac{|\mathcal{H}|-E}{B^{[K:\Q]/|G|}(\log B)^\gamma}.$$
\end{theorem}

\subsection{A unified version of theorem A and theorem B.} 

Retain the notation and assumptions of \S 2. Fix an affine model $P(T,Y) \in O_K[T,Y]$ of $F/K(T)$; note that $P$ is monic in $Y$. If $\Delta_P(T)$ is the discriminant of $P$ relative to $Y$, set $\delta_P = \deg(\Delta_P(T))$. Fix $\delta >\delta_P$. As in \cite{debconjmalle}, one can take $\delta = 3r |G|^4 \log(|G|)$.
 \medskip

Given a finite set $S$ of primes of $K$ and a Frobenius data $\mathcal{F}$ on $S$, let $N(F/K(T),y,\mathcal{F})$ be the number of distinct Galois extensions $F_{t_0}/K$ of group $G$ obtained by specialization from $F/K(T)$ at some $t_0 \in K$, with discriminant of norm $|N_{K/\Q}(d_{F_{t_0}/K})|\leq y$ and such that for every $\ip \in S$, $F_{t_0}/K$ is unramified in $\ip$ and ${\rm Frob}_{\ip}(F_{t_0}/K)\in \frobp$. \smallskip

We say that a prime $\ip$ of $K$ is {\it good} for $F/K(T)$ if $\ip$ does not divide $|G|$, the branch divisor $\mathbf{t}=\{t_1,\cdots, t_r\}$ is étale at $\ip$ and there is no vertical ramification at $\ip$. We say $\ip$ is {\it bad} otherwise (we refer to \cite{debgazi} and \cite{legrandspec} for precise definitions). We will use that there exist only finitely many bad primes. \smallskip

The constant $p_0$ in theorem AB below is the one that appears in theorem \ref{SPECTHM}.

\vskip 2mm
\noindent \textbf{Theorem AB.}  \textit{ For every number $y>0$, consider the set $S_y$ of primes $\ip$ of $K$ over some prime $\displaystyle p\in [p_0,\frac{\log y}{2 \dd \delta}]$ that are good for $F/K(T)$. If $y$ is suitably large (depending on $F/K(T)$, $\delta$), then for every Frobenius data $\frob_y$ on $S_y$, we have $$N(F/K(T),y,\frob_{y}) \geq y^{(1-1/|G|)/\delta}.$$}
\vskip 2mm

\subsection{Proof of theorem AB assuming theorems \ref{SPECTHM} and \ref{HilbertMalle}.} \label{preuve26}
Theorem \ref{SPECTHM} produces many {\og good \fg} specialization points $t_0$ with arbitrarily bounded height $H(t_0)$. We explain below how to bound $H(t_0)$ in terms of some given number $y>0$ to fullfill the required condition $|N_{K/\Q}(d_{F_{t_0}/K})|\leq y$. \smallskip

Set $\delta^-=\displaystyle \frac{\delta + \delta_P}{2}$ (we have $\delta_P<\delta^-<\delta$) and $B=y^{1/ \dd \delta^-}$.

\begin{prop} \label{Benfctdey}
For $y$ suitably large, the specializations $F_{t_0}/K$ of $F/K(T)$ at $t_0\in O_K$ such that $\Delta_P(t_0)\neq 0$, $H(t_0) \leq B$ and $F_{t_0}/K$ is Galois of group $G$ satisfy $|N_{K/\Q}(d_{F_{t_0}/K})|\leq y.$
\end{prop}

\begin{proof}
The polynomial $P(t_0,Y)$ is in $O_K[Y]$ (as $t_0 \in O_K$), is monic, irreducible in $K[Y]$ and of degree $|G|$. Hence, if $y_0 \in \overline{K}$ is a root of $P(t_0,Y)$, then $1,y_0,\cdots,y_0^{|G|-1}$ is a $K$-basis of $F_{t_0}/K$ consisting of elements in $O_{F_{t_0}}$. Thus $$d_{F_{t_0}/K} \mid {\rm disc}(1,y_0,\cdots,y_0^{|G|-1})={\rm disc}(P(t_0,Y))={\rm disc}_Y(P(T,Y))_{T=t_0}=\Delta_P(t_0)$$
We deduce $$|N_{K/\Q}(d_{F_{t_0}/K})|\leq |N_{K/\Q}(\Delta_P(t_0))|.$$ 

Straightforward estimates involving norms and height show next that 
$$|N_{K/\Q}(\Delta_P(t_0))| \leq C B^{\dd \delta_P}$$
for some constant $C>0$ depending on $P$ and $K$; these estimates are detailed in \S 3.1. Hence we obtain :
$$ |N_{K/\Q}(d_{F_{t_0}/K})| \leq C B^{\dd \delta_P},$$

The $\log$ of this last term is $$\log [C B^{\delta_P \dd}] \sim \frac{\dd \delta_P}{ \dd \delta^-}\log y.$$ As $\delta_P< \delta^-$, conclude that for $y$ suitably large in terms of $F/K(T)$ and $ \delta$, we have $$|N_{K/\Q}(d_{F_{t_0}/K})|\leq y.$$
\end{proof}

We will apply theorem \ref{SPECTHM} with $B=y^{1/\dd\delta^-}$ and theorem \ref{HilbertMalle} with the following choice of the set $\mathcal{H}$ : the set of $t_0\in O_K$ satisfying the conclusions of theorem \ref{SPECTHM} with $B=y^{1/\dd \delta^-}$. We can now proceed to the proof of theorem AB.%\ref{malle}.
\medskip

As $\delta^{-} < \delta$, by the choice of $B$, we have $\displaystyle [p_0, \frac{\log y}{2 \dd \delta}] \subset [p_0,\log B /2]$. Fix a Frobenius data $\frob_y$ on $S_y$ and extend it in an arbitrary way to a Frobenius data on $S_B \supset S_y$ of all the primes of $K$ over the interval $[p_0, \log B /2]$.
\smallskip

According to theorem \ref{SPECTHM}, we have $|\mathcal{H}| \geq \displaystyle  \frac{B^{\dd}}{c^{\log B/\log \log B}}$. From theorem \ref{HilbertMalle}, there exist $E,\gamma \geq 0$ depending on $F/K(T)$ such that for $y$ suitably large,
\begin{align*}
\mathcal{N}(B,\mathcal{H}) & \geq \frac{|\mathcal{H}|-E}{B^{\dd /|G|} (\log B)^{\gamma}} \\ 
& \geq \displaystyle \frac{ B^{\dd - \dd /|G|}}{ (\log B)^{\gamma} \hskip 2pt c^{\log B/\log \log B}} - \frac{E}{B^{\dd /|G|} (\log B)^{\gamma}}
\end{align*}
%où $c=\displaystyle  \frac{\chi(\mathcal{F}_x)}{\left(2\Pi(p_{-1})\right)^{\dd}} \times \left(\frac{1}{2r|G|}\right)^{|S_x|}$. Puis 

Denote the last lower bound by $f(B)$. The logarithm of $f(B)$ is asymptotic to \\ $\dd (1-1/|G|)\log B$. From the choice of $B$, we finally obtain

$$\log(f(B)) \sim \frac{\delta}{\delta^-}\log(y^{(1-1/|G|)/\delta}).$$

\noindent Because $\delta>\delta^-$, we obtain that for $y$ suitably large $\log(f(B)) > \log(y^{(1-1/|G|)/\delta})$ and so $$\mathcal{N}(B,\mathcal{H}) \geq y^{(1-1/|G|)/\delta}.$$

The inequality $N(F/K(T),y,{\rm \frob}_{y}) \geq \mathcal{N}(B,\mathcal{H})$ concludes the proof of theorem AB.%\ref{malle}. 
\qed

\begin{rmq}
Our counted extensions are obtained by specialization of one single regular extention $F/K(T)$. There may be other extensions $E/K$ (not coming from $F/K(T)$ by specialization) satisfying the same conditions. This explains why our constant $\displaystyle \alpha(G)=\frac{1-1/|G|}{\delta}$ is smaller than the Malle constant $a(G)$ (see \cite[lemma 4.1]{debconjmalle}).
\end{rmq}

\subsection{Proof of theorems A and B assuming theorem AB.} Concerning theorem A, one proceeds as follows. Classically, every finite group $G$ is known to be a regular Galois group over some number field, say $K_0$. If $K$ is a number field containing $K_0$, $G$ is still a regular Galois group over $K$. Clearly $N(K,G,y)$ from \S 1.1 is bigger than $N(F/K(T),y,\frob_{y})$ from theorem AB. %\ref{malle}
Thus theorem A %\ref{lowconjmalle}
(with $\alpha(G)=(1-1/|G|)/{\delta}$) follows immediately from theorem AB.%\ref{malle}.
\medskip

To prove theorem B, %\ref{resultgrunwald}
 suppose first that $G$ is a regular Galois group over $K$ and fix a $K$-regular Galois extension $F/K(T)$ of group $G$. Consider an unramified Grunwald problem $(G,(L^{\ip}/K_{\ip})_{\ip\in S})$. For each $\ip \in S$, let $\frobp$ be the conjugacy class in $G$ of the Frobenius of $L^{\ip}/K_{\ip}$ (which generates ${\rm Gal}(L^{\ip}/K_{\ip})$). Then for a Galois extension $L/K$ of group $G$, unramified at $\ip$, we have $LK_{\ip}/K_{\ip}=L^{\ip}/K_{\ip}$ if and only if ${\rm Frob}_{\ip}(L/K)\in \frobp$. Theorem B %\ref{resultgrunwald}
(in this first case) then follows from theorem AB.%\ref{malle}. 
\smallskip

Namely, the set $S_{exc}$ can be chosen as the set of primes $\ip$ of $K$ such that either $\ip$ is over some prime number $p\in [2,p_0[$ \footnote{This interval does not depend on the Grunwald problem $(G,(L^{\ip}/K_{\ip})_{\ip\in S})$.} or $\ip$ is bad for $F/K(T)$. Here $p_0$ is the prime number defined in \S 2.1 from the group $G$, the branch point number $r$ of $F/K(T)$ and the genus $g$ of $F$. Given a set $S$ of primes of $K$ such that $S \cap S_{exc} = \emptyset$, take $y$ suitably large so that the interval $\displaystyle [p_0,\frac{\log y}{2 \dd \delta^-}]$ contains all prime numbers under all primes of $S$. Applying theorem AB with letting $y$ go to $\infty$ yields infinitely many extensions $L/K$ that are solution to any Grunwald problem $(G,(L^{\ip}/K_{\ip})_{\ip \in S})$.
\medskip

Consider now the general case, i.e., $G$ is not necessarily a regular Galois group over $K$. The definition of $S_{exc}$ relies on results from \cite{debgazi}. A constant $c(G)$ is defined there, for which the following lemma is true.

\begin{lemme}\label{lemmegrunwald}
Given a finite group $G$ and a number field $K$, there exist non negative integers $r$ and $g$ such that with 
$$S_{exc}= \{\ip \text{ prime of } K \hskip 2pt \vert \hskip 2pt p_{\ip} \mid 6|G| \text{ or } p_{\ip} \leq \max(p_0,c(G)) \}$$ 
the following holds. For every finite set $S$ of primes of $K$ with $S \cap S_{exc} =\emptyset$, there exists a finite Galois extension $M/K$ totally split at each prime $\ip \in S$ and a $M$-regular Galois extension $F/M(T)$ of group $G$ such that $F/M(T)$ has $r$ branch points, the genus of $F$ is $g$ and each prime $\mathcal{P}$ of $M$ over a prime $\ip \in S$ is good for $F/M(T)$.
\end{lemme}

Here $p_0$ is the prime number defined in \S 2.1 from $K,G$ and the integers $r$, $g$ from the statement.
\smallskip

A proof of this lemma is given in \S 5 of \cite{debgazi}.
\medskip

As in theorem B, let then $(G,(L^{\ip}/K_{\ip})_{\ip\in S})$ be an unramified Grunwald problem  over $K$ with $S \cap S_{exc} =\emptyset$. Let $M/K$ be the extension given by lemma \ref{lemmegrunwald} for this $S$. Consider next the Grunwald Problem over the field $M$ deduced by the base changes $M_{\mathcal{P}}/K_{\ip}$, $\ip \in S$. The first case applied with $(G,(L^{\ip}M_{\mathcal{P}}/M_{\mathcal{P}})_{\mathcal{P}\in S_M})$ produces an infinite number of $M$-solutions to the Grunwald problem $(G,(L^{\ip}/K_{\ip})_{\ip\in S})$. More specifically, note that if $\mathcal{P}\in S_M$, then $\mathcal{P}$ is unramified in $M/\Q$, $p_{\mathcal{P}} > p_0(r,g,G)=p_0(F/M(T))$ (because $S \cap S_{exc}=\emptyset$) and $\mathcal{P}$ is good for $F/M(T)$ (from lemma \ref{lemmegrunwald}): thus if $\mathcal{P} \in S_M$, $\mathcal{P}$ is not in the exceptional set of the first case for $F/M(T)$. This proves theorem B.

\subsection{A generalization of theorem AB to not necessarily Galois extensions.}
Denote by $S_n$ the permutation group on $n$ letters $1,\cdots,n$. For an extension $E/K$ of degree $n$, we denote by $\hat{E}/K$ its Galois closure. The Galois group ${\rm Gal}(\hat{E}/K)$ acts transitively on the $n$ embeddings $E\hookrightarrow \overline{K}$. Let $G(1)\subset G$ be the stabilizing subgroup of the neutral element $1$. We say that the extension $E/K$ has {\it Galois group $G\subset S_n$} if $G$ is the Galois group of $\hat{E}/K$ and $E$ is the fixed field of $G(1)$ in $\hat{E}$. Consider the number 

$$N(K,G \subset S_n,y)=\#\{E/K \mid E/K \textit{ of Galois group } G\subset S_n,\hskip 2pt |N_{K/\Q}(d_{E/K})|\leq y\}.$$  

\begin{theorem}\label{generalcase}
If $G$ is a regular Galois group over $K$, then there exists $\alpha >0$ such that  $$N(K,G \subset S_n,y)\geq y^{\alpha} \text{ for every suitably large } y$$
where $\alpha =(1-1/|G|)/{\delta}$ with $\delta > \delta(P)$ and $P(T,Y)$ an affine model of some $K$-regular Galois extension $F/K(T)$ of group $G$.
\end{theorem}

\begin{rmq}
The special case $G(1)=\{1\}$ corresponds to the case the action $G \subset S_n$ is free and transitive, equivalently the extension $E/K$ is Galois of group $G$. As the proof below shows, we will deduce the general case from this special case.
\end{rmq}

\begin{proof}
To every Galois extension $N/K$ of group $G$ corresponds one intermediate extension $E/K$ which satisfies $E=N^{G(1)}$ and $\hat{E}=N$. Furthermore, this extension is of degree $n$.
\smallskip

%Indeed, as $G$ is a transitive permutation group, the orbit of $\{1\}$ is of cardinal $n$. La formule des classes implique que $G(1)$ est d'indice $n$. L'extension $N^{G(1)}/K$ fixée par $G(1)$ est donc de degré $n$.\smallskip
The only point to be checked is that $\hat{E}=N$. Let $H$ be the biggest normal subgroup of $G$ that is contained in $G(1)$. The Galois closure of $N^{G(1)}$ is $N^{H}$. But as $H$ is a normal subgroup, for every $g\in G$, we have $H=H^g \subset G(1)^g=G(i)$ (where $i=g(1)$). Thus, we have $ H \subset \displaystyle \bigcap_{i=1}^n G(i)$ and $H$ is then the trivial subgroup. Hence $N^{H}=N$.
\smallskip

The proof of theorem \ref{generalcase} easily follows : the Galois extensions $N/K$ provided by theorem A (or by theorem AB) provide as many extensions $E/K$ as requested in the general case. Note that the norm $N_{K/\Q}(d_{E/K})$ is less than $N_{K/\Q}(d_{N/K})$.
\end{proof}

\section{Proof of the diophantine theorem C and corollary C} \label{sectiondiophantine}
In this section, we prove theorem C and corollary C. We work over a fixed number field $K$ of degree $[K:\Q]=\dd$.

\subsection{Basic data and generalized Heath-Brown result.}

\subsubsection{The height.}
Recall that $M_K$ is the set of places of $K$ and for $v\in M_K$, denote by $K_v$ the completion of $K$ for $v$, by $O_v$ its valuation ring, and by $\Q_v$ the completion of $\Q$ for $v$ ($\Q_p$ for a finite place and $\R$ for an archimedean place). The places are normalized in such a way they are equal to the usual absolute value on $\Q_v$. We denote by $\dd_v$ the degree $[K_v:\Q_v]$.
\smallskip

The height of $x \in O_K$ is $H(x)=\displaystyle \max_{\sigma : K\hookrightarrow \overline{K}} |\sigma(x)|=\max_{\substack{v\in M_K \\ v/\infty}}|x|_v.$

\begin{rmq}
If $x\neq 0$, then $H(x)=\displaystyle \max_{\sigma : K\hookrightarrow \overline{K}} \max(1,|\sigma(x)|)=\max_{\substack{v\in M_K \\ v/\infty}}(1,|x|_v)$. \\
Indeed, it is classical that if $x\in O_K$, $x\neq 0$ and $H(x)\leq 1$ then $x$ is a $k$-th root of $1$ for some $k\geq 1$.
\end{rmq}

We generalize the height to elements of $K$ and to tuples as follows :
\begin{itemize}
\item for $x\in K$, $H(x)=\displaystyle \max_{v\in M_K} \max(1,|x|_v)$
\item for $\underline{a}=(a_1,\cdots,a_n) \in K^n$, $H(a)= \displaystyle \max_{v\in M_K} H_v(\underline{a})$, where $H_v(\underline{a})=\displaystyle \max(|a_1|_v,\cdots,|a_n|_v)$, \\ and $H^+(\underline{a})=\displaystyle \max_{v\in M_K} H_v^+(\underline{a})$, where $H_v^+(\underline{a})=\displaystyle \max(1,|a_1|_v,\cdots,|a_n|_v)$
\end{itemize}

Note that for $x\in O_K$, $H(x)=H(1,x)=H^+(x)$. \smallskip

The height of a polynomial $P$ with coefficients $a_1,\cdots,a_n$ in $K$ is $H(P)=H(a_1,\cdots,a_n)$.

\subsubsection{Preliminary lemmas.}
The following notation and properties are used all along this section. For theorem C, we consider a polynomial $F(X_1,X_2) \in O_K[X_1,X_2]$. \smallskip

For corollary C, we prefer to denote the indeterminates by $T$ and $Y$, as they do not play the same role. \smallskip

 Both polynomials are assumed to be irreducible over $K$. We let
\begin{itemize}
\item $m$ be the degree of $F$ in $X_1$ (or in $T$),
\item $n$ be the degree of $F$ in $X_2$ (or in $Y$),
\item $d$ be the total degree of $F$ (we may and will assume that $d\geq 2$).
%\item $d_K$ is the discriminant of $K/\Q$
%\item $\Delta_K$ the discriminant of $K/\Q$
% \item pour $v$ une valuation sur $K$, $\delta_v$ le degré $[K_v:\Q_v]$
%\item $H(x)$ is the Weil height of an algebraic number $x$ (or a tuple)
%\item $h=\max (H(F),e^e)$ ($H(F)$ is the height of the polynomial $F$)\\
\end{itemize} 
\medskip

The following statement collects different properties used in this paper.
\begin{prop} \label{hauteval}
Let $\underline{a}=(a_1,\cdots,a_m)$ be a $m$-tuple in $K^m$ ($m\in \N$), let $P(\underline{X}) \in K[\underline{X}]=K[X_1,...,X_n]$ be a polynomial in $n$ variables with $l$ non-zero coefficients and $(x_1,\cdots,x_n) \in K^n$. Let $t\in K$, $t\neq 0$, and $\sigma : K \rightarrow \overline{\Q}$ be a field morphism. Then we have
\begin{enumerate}
%\item $H(t.\underline{a})=H(\underline{a})$ and $H(t)=H(1/t)$. FAUX AVEC LA MAISON
\item $H(\underline{a})=H(\sigma(\underline{a}))$.
\item $ H(a_i) \leq H^+(\underline{a})\leq H(a_1)\cdots H(a_m). ~~ (i= 1,...,m).$
\item $H(P(x_1,\ldots,x_n)) \leq \displaystyle l.H^+(P).\prod_{i=1}^n H(x_i)^{\deg_{X_i}(P)}.$
\end{enumerate}
\end{prop}

\begin{proof}
$1.$ is clear.
\smallskip

$2.$ Using the definition, we have $$H(a_i) = \max_{v\in M_K} \max(|a_i|_v) \leq \max_{v\in M_K} \max(1,|a_1|_v,...,|a_n|_v)\leq \prod_{i=1}^m \max_{v\in M_K} \max(1,|a_i|_v).$$

$3.$ We write $\displaystyle P(\underline{X})=\sum p_{\underline{a}}\underline{X}^{\underline{a}}=\sum p_{a_1,...a_n}X_1^{a_1}...X_n^{a_n}$ and set $d_i=\deg_{X_i}(P)$, $i=1,\cdots n$.
\begin{itemize}
\item for every archimedean place $v$, we have $$|P(\underline{x})|_v \leq l.\max_{\underline{a}}(|p_{\underline{a}}|_v) \max(1,|x_1|_v)^{d_1}...\max(1,|x_n|_v)^{d_n}.$$
\item for every finite place $v$, we have $$|P(\underline{x})|_v \leq \max_{\underline{a}}(|p_{\underline{a}}|_v) \max(1,|x_1|_v)^{d_1}...\max(1,|x_n|_v)^{d_n}.$$
\end{itemize}

Whence 
\begin{align*}
H(P(\underline{x}))&=\max_{v\in M_K} \max(1,|P(\underline{x})|_v)\\
& \leq l. \max_{v\in M_K} \max_{\underline{a}}(1,|p_{\underline{a}}|_v) . \max_{v\in M_K} \max(1,|x_1|_v)^{d_1} \cdots\times \max_{v\in M_K}\max(1,|x_n|_v)^{d_n}\\
\end{align*}
\end{proof}

\textbf{Ideals in $O_K$ and norm}. The norm of an ideal $J\subset O_K$ is the cardinality of the quotient ring $N_{K/\Q}(J)=\#O_K/J$ (see \cite{samuel} for more details).
\smallskip
For $a\in O_K$, $$N_{K/\Q}(aO_K)=|N_{K/\Q}(a)|=\displaystyle \prod_{\sigma : K \rightarrow \overline{\Q}}|\sigma(a)| \leq H(a)^{\dd}.$$

We can now prove the inequality 
$$ |N_{K/\Q}(\Delta_P(t_0)| \leq C B^{\dd \delta_P} \text{ if } H(t_0)\leq B$$ 
stated in the proof of proposition \ref{Benfctdey}.

\begin{proof}
Using the inequality between norm and height, we obtain : 
$$|N_{K/\Q}(\Delta_P(t_0))| \leq H(\Delta_P(t_0))^{\dd}$$
Then, as $\Delta_P$ is a polynomial of degree $\delta_P$ and has at most $1+\delta_P$ non-zero coefficients. Proposition \ref{hauteval} $(3)$ yields
\begin{align*}
|N_{K/\Q}(\Delta_P(t_0))| & \leq [(1+\delta_P)H^+(\Delta_P)H(t_0)^{\delta_P}]^{\dd} \\
& \leq (1+\delta_P)^{\dd} H^+(\Delta_P)^{\dd} B^{\delta_P \dd}.
\end{align*}
\end{proof}

We will also use the following result.
\begin{lemme}\label{normid}
Let $a \in O_K$. The number of primes $\mathfrak{p}$ of $K$ which divide the ideal $aO_K$ is less than or equal to $\dd \log_2(H(a)).$
\end{lemme}
\begin{proof}
Let $\ip_1,...,\ip_n$ be the prime ideals of $O_K$ dividing $aO_K$. As $O_K$ is a Dedekind domain, we have $aO_K=\ip_1^{\alpha_1}...\ip_n^{\alpha_n}$, where $\alpha_1, \cdots,\alpha_n$ are positive integers. Then 
$$H(a)^{\dd} \geq |N_{K/\Q}(a)|=\prod_{i=1}^nN_{K/\Q}(\ip_i)^{\alpha_i}.$$
As $N_{K/\Q}(\ip_i) \geq 2$, $i=1,\cdots,n$, we obtain $H(a)^{\dd} \geq 2^n$, thus proving the lemma.
\end{proof}

%%%%%%%%%%%%%%%%%%%%%%%%%%%%%%%%%%%%%%%%%%%%%%%%%%%%%%%%%%%%%%%%%%%%%%%%%%%%%%%%%%%%%%%%%%
%%%%%%%%%%%%%%%%%%%%%%%%%%%%%%%%%%%%%%%%%%%%Partie 2%%%%%%%%%%%%%%%%%%%%%%%%%%%%%%%%%%%%%%
%%%%%%%%%%%%%%%%%%%%%%%%%%%%%%%%%%%%%%%%%%%%%%%%%%%%%%%%%%%%%%%%%%%%%%%%%%%%%%%%%%%%%%%%%%%

\subsubsection{A generalized Heath-Brown result.}
For every real number $B>0$, set
$$R(F,B)=\{(x_1,x_2) \in O_K^2 ~ \vert ~ F(x_1,x_2)=0,~ H(x_1)\leq B ,~ H(x_2)\leq B \},$$ 
and $$N(F,B)=\#R(F,B).$$

Our approach for bounding the number $N(F,B)$ follows an idea of Heath-Brown \cite{heathbrown1} which Walkowiak made effective (both in the case $K=\Q$).  We generalize to the case of an arbitrary number field $K$. The method consists in splitting the set $R(F,B)$ in $k$ subsets, each being the zero set of some polynomial $F_i\in O_K[X_1,X_2]$ relatively prime with $F$, $i=1,\cdots,k$. The Bezout theorem then yields the desired bound for $N(F,B)$. An important point is to have a good upper bound for the number $k$ of polynomials $F_i$. To this end, we prove the following effective generalized Heath-Brown result.

\begin{theorem} \label{HBE}
Let $F(X_1,X_2) \in O_K[X_1,X_2]$ be a polynomial, irreducible in $\overline{K}[X_1,X_2]$ of degree $d$, let $D\geq d$ be an integer and let $B$ be a suitably large real number (depending on $\dd$, $d$, $D$). There exist a number $k\geq 1$ and some polynomials $F_1,...,F_k \in O_K[X_1,X_2]$ relatively prime with $F$ in $\overline{K}[X_1,X_2]$ and of degree $\deg(F_i)\leq D$, such that every point $(x_1,x_2) \in R(F,B)$ is a zero of at least one of $F_1,\ldots,F_k$. Furthermore, the integer $k$ is bounded from above by :
$$ c_2 \hskip 2pt d^3 \log^2(d^3H^+(F)B^{d-1}) \hskip 2pt (B^{d^{-1}+6D^{-1}})^\dd $$ where $c_2$ is a constant depending only on $K$.
\end{theorem}

\begin{rmq}
This theorem is true for all integers $D\geq d$. In the proof of theorem C, we will use it with $D=[d\log(B)+1]$ (where $[.]$ is the integral part of a real number).
\end{rmq}

\subsection{Proof of theorem \ref{HBE}} \label{sectionpreuveHB}
Fix $D \geq d$ and $B>0$. The condition that $B$ should be suitably large appears in \S3.2.3. We explain below how to construct the polynomials $F_1, \cdots , F_k$, that appear in theorem \ref{HBE}.
\smallskip

We take one of the polynomials $F_i$, $i=1,\cdots,k$ to be $\displaystyle \frac{\partial F}{\partial X_1}$; it is relatively prime to $F$ and of degree $\leq d$. So we may next focus on the subset  
$$S(F,B)=\{\underline{x} \in R(F,B) \mid ~ \frac{\partial F}{\partial X_1}(\underline x)\neq 0\} \subset R(F,B),$$
and look for $k'$ polynomials $F_i$ to cover this subset. The number $k$ in theorem \ref{HBE} will be equal to $1+k'$.

\subsubsection{First step : constructing subsets $S(F,B,\ip) \subset S(F,B)$.}
Let $\ip$ be a prime ideal of $O_K$ and $$S(F,B,\ip)=\{\underline{x} \in S(F,B) \mid \displaystyle \frac{\partial F}{\partial X_1}(\underline x) \notin \ip \}.$$ We have $$S(F,B)=\displaystyle \bigcup_{\ip \text{ prime of } K}S(F,B,\ip).$$
The following lemma shows that one can take finitely many primes $\ip$ in the previous union and that these primes can be chosen to be totally split in $K/\Q$.

%%%%%%%%%%%%%%%%%%%%%%%%%%%%%%%%%%%%%%%%%%%%%%%%%%%%%lemme 2.4 %%%%%%%%%%%%%%%%%%%%%%%%%%%%%%%%%%%%%%%%%%%%%%%%%%%%%%%%%%%%%%%%%%%%%%%%%%%%%%%%%%%%%%%%
\begin{lemme}\label{lemmeSFBP}
Let $P$ be an integer, $h(B)=\log_2(d^3H^+(F)B^{d-1})$ and $r=[\log_2 (\dd h(B))]+1$. %(\textit{la notation [.] désignant la partie entière}). 
Then for $P$ suitably large (depending on $B$, $K$), there exist $r$ totally split prime ideals $\ip_1,...,\ip_{r}$ of $K$ such that $\displaystyle S(F,B)= \bigcup_{i=1}^{r} S(F,B,\ip_i)$ and for which we furthermore have

$$P\leq N_{K/\Q}(\ip_i)\leq C_1 h(B)^2 \hskip 3pt \frac{P}{\log P} \hskip 2pt \log(\frac{P}{\log P})$$ 

\noindent for a constant $C_1$ depending on $K$.
\end{lemme}
\begin{proof} Fix $\underline{x}=(x_1,x_2)\in S(F,B)$; we have $\displaystyle \frac{\partial F}{\partial X_1}(\underline x)\neq 0$. The number of prime ideals $\ip$ of $O_K$ such that $\displaystyle \frac{\partial F}{\partial X_j}(\underline x) \in \ip $ is at most $\dd\log_2(H(\displaystyle \frac{\partial F}{\partial X_1}(\underline x)))$ from lemma \ref{normid}. The height of $\displaystyle \frac{\partial F}{\partial X_1}(\underline x)$ can be estimated using proposition \ref{hauteval} : the number $l$ of the non-zero monomials of $ \displaystyle \frac{\partial F}{\partial X_1}$ is bounded by $d(d+1)/2 \leq d^2$, its degree by $d-1$ and its height by $dH^+(F)$, whence 

$$H(\frac{\partial F}{\partial X_j}(x_1,x_2)) \leq l.H^+(\frac{\partial F}{\partial X_j}).B^{d-1} \leq d^3H^+(F)B^{d-1}.$$

Consider the Galois closure $\widehat{K}/\Q$ of $K/\Q$ and its Galois group $\Gamma$. For a conjugacy class $C$ of $\Gamma$ and $x\geq 1$, denote by $\pi_C(x)$ the number of prime numbers $p\leq x$, which do not ramify in $\widehat{K}/\Q$ and such that ${\rm Frob}_{p}(\widehat{K}/\Q) \in C$. For $C=\{1\}$, $\pi_{\{1\}}(x)$ is the number of primes $p\leq x$, totally split in $O_{\widehat{K}}$.
\medskip

Fix an integer $P>0$ and let $x>0$ such that 
$$(*) \hskip 10pt \pi_{\{1\}}(x)\geq h(B)+1+\pi_{\{1\}}(P).$$ 

\noindent More precisely, we take $x$ as follows : $x=2a\log a$ with $a=6 \hskip 2pt|\Gamma| \hskip 2pt h(B) \pi_{\{1\}}(P)$. For $P$ suitably large, depending on $K$, $B$, it is easily checked that $\frac{x}{\log x} \geq a$ and so $$ \frac{x}{2|\Gamma|\log x} \geq 3 \hskip 2pt h(B) \hskip 2pt \pi_{\{1\}}(P) \geq h(B) +1+ \pi_{\{1\}}(P),$$
which implies that $\pi_{\{1\}}(x) \geq h(B)+1+\pi_{\{1\}}(P)$ as from classical results on the Chebotarev theorem, we have $\pi_{\{1\}}(x) \geq \frac{x}{2|\Gamma|\log x},$ for $x$ suitably large. \medskip

From $(*)$, there exist at least $[h(B)]+1$ prime numbers $p\in]P,x]$ that are totally split in $K/\Q$. Every such prime number $p$ provides $\dd$ primes of $K$ of norm equal to $p$. Hence we have $\dd[h(B)]+\dd$ distinct prime ideals totally split in $K$ and of norm $\leq x$. Furthermore, this number $\dd[h(B)]+\dd$ is $\geq r$ as $r=[\dd h(B)]+1$.

\indent We choose $r$ of these ideals which we denote by $\ip_1,...,\ip_r$. By lemma \ref{normid}, there exists an ideal, say $\ip_i$ ($i\in \{1,...,r\}$), such that $\displaystyle \frac{\partial F}{\partial X_1}(x_1,x_2) \notin \ip_i$, which means that $(x_1,x_2)\in S(F,B,\ip_i)$. Thus we obtain $\displaystyle S(F,B)= \bigcup_{i=1}^{r} S(F,B,\ip_i)$. \smallskip

Now $\pi_{\{1\}}(P) \leq \displaystyle \frac{2P}{|\Gamma| \log P}$ for $P$ suitably large depending on $K$, $B$. So for $i = 1,\cdots,r$, 
$$N_{K/\Q}(\ip_i)\leq x=2a\log a \leq 12|\Gamma| \hskip 2pt h(B) \hskip 2pt \frac{2 P}{|\Gamma|\log P} \hskip 2pt \log(6 |\Gamma| h(B) \frac{2 P}{|\Gamma| \log P}).$$

\noindent We conclude that for some constant $C_1$ depending on $K$ we have 
$$N_{K/\Q}(\ip_i)\leq C_1 h(B)^2 \frac{P}{\log P} \ln (\frac{P}{\log P}).$$
\end{proof}

\subsubsection{Working on $S(F,B,\ip)$ for a fixed $\ip \in \{\ip_1,\cdots,\ip_r\}$.}

For the next steps, we choose a monomial $X_1^{m_1} X_2^{m_2}$ such that the corresponding coefficient in $F$ is non-zero and such that $m_1+m_2=d$. We let then $\mathcal{E}$ be the following set of monomials 
$$\mathcal{E}= \{X_1^{e_1} X_2^{e_2} \mid 0\leq e_i\leq m_i,~ i=1,2, ~ e_1+e_2\leq D\}.$$

Fix $\ip \in \{\ip_1,\cdots\ip_r\}$. Recall from lemma \ref{lemmeSFBP} that $N_{K/\Q}(\ip)=p \geq P$ and $D \geq d$. Let $\F_p=O_K/\ip$ be the residue field of $\ip$. For every point $\underline{t}=(t_1,t_2)\in\F_p^2$ of $F$ modulo $\ip$ such that $\displaystyle \frac{\partial F}{\partial X_1}(\underline{t}) \not\equiv 0 \mod \ip$, consider then the following subset of $S(F,B,\ip)$ : 
$$S(\underline{t})=\{(x_1,x_2) \in S(F,B,\ip) : x_i  = t_i \mod \ip,~i=1,2\}.$$

\noindent We have $S(F,B,\ip)=\displaystyle \bigcup_{\underline{t}}S(\underline{t})$ where $\underline{t}$ ranges over the all non-singular points of $F$ modulo $p$. Fix such a $\underline{t}$. We have $\frac{\partial F}{\partial X_1}(\underline{t})\neq 0  \mod \ip$. The goal now is to construct one polynomial (one of those in theorem \ref{HBE}) that vanishes at all points of $S(\underline{t})$. \smallskip

 Denote by $\underline{x_{i}}=(x_{i1},x_{i2})$, $i=1,\cdots ,L$, the elements of $S(\underline{t})$ (with $L={\rm card}(S(\underline{t})$). 
\smallskip

 Set $E=\#\mathcal{E}$ and let $M$ be the $L\times E$ matrix
$$M=(\underline{x_i}^{\underline{e}})_{1\leq i \leq L,~ \underline{e}\in \mathcal{E}} .$$
More specifically, if $\mathcal{E}=\{\underline{X}^{\underline{e_1}},...,\underline{X}^{\underline{e_E}}\}$,
$$ M= \begin{pmatrix}
   \underline{x_1}^{\underline{e_1}} &  \cdots & \underline{x_1}^{\underline{e_E}}\\
   \underline{x_2}^{\underline{e_1}} &  \cdots & \underline{x_2}^{\underline{e_E}}\\
   \vdots & \vdots & \\
   \underline{x_L}^{\underline{e_1}}  & \cdots & \underline{x_L}^{\underline{e_E}} 
\end{pmatrix} 
= \begin{pmatrix}
   x_{11}^{e_{11}}x_{12}^{e_{12}} &  \cdots & x_{11}^{e_{E1}}x_{12}^{e_{E2}}\\
   x_{21}^{e_{11}}x_{22}^{e_{12}} &  \cdots & x_{21}^{e_{E1}}x_{22}^{e_{E2}}\\
   \vdots & \vdots & \\
   x_{L1}^{e_{11}}x_{L2}^{e_{12}}  & \cdots & x_{L1}^{e_{E1}}x_{L2}^{e_{E2}} 
\end{pmatrix}$$

\noindent Finally set $E'= \displaystyle \sum_{i=1}^E e_{i1}+e_{i2}$.
\begin{prop} \label{p1}
Assume that $P^{E(E-1)/2}\geq (E^E B^{E'})^{\dd}$. Then we have the following.
\begin{enumerate}
\item The rank of $M$ is $\leq E-1$,
\item There exists a non-zero polynomial $F_{\underline{t}}$ of degree $\leq D$ such that
\begin{itemize}
\item $F_{\underline{t}}(\underline{x})=0$ for all $\underline{x}\in S(\underline{t})$,
\item $F_{\underline{t}}$ and $F$ are relatively prime.
\end{itemize}
\end{enumerate}
\end{prop}

\indent The proof uses the following lemma which is some version of Hensel's lemma and reduces the problem from two to one variable. We refer to \cite{walarticle} lemma 1.2 for a proof (mod $p^m$ there just has to be replaced by mod $\ip^m$).

\begin{lemme}\label{LFI}
Let $F(X_1,X_2) \in \widetilde{O_K}[X_1,X_2]$ be a polynomial in two variables with coefficients in the completion $\widetilde{O_K}$ of $O_K$ for the prime ideal $\ip$. Let $\underline{u}=(u_1,u_2) \in \widetilde{O_K}$ such that $F(\underline{u})=0$ and $\frac{\partial F}{\partial X_1}(\underline{u})\notin \ip$. For every integer $m\geq 1$, there exists $f_m(Y)\in \widetilde{O_K}[Y]$ such that if $F(\underline{x})=0$ for a certain $\underline{x}=(x_1,x_2)\in \widetilde{O_K}^2$ with $\underline{x}=\underline{u} \mod \ip$, then $x_1=f_m(x_2) \mod \ip^m$ (for every $m\geq 1$).
\end{lemme}

\begin{proof}[Proof of proposition \ref{p1}]

First, note that $2.$ easily follows from $1.$ As the rank of $M$ is $\leq E-1$, there exists a non-zero matrix $C=(c_{\underline{e}}) \in O_K^E$, such that $M C=0$. We use this matrix to construct our polynomial $F_{\underline{t}}$ :
$$F_{\underline{t}}(X_1,X_2)=\sum_{\underline{e}\in\mathcal{E}}c_{\underline{e}}X_1^{e_1}X_2^{e_2}.$$
This non-zero polynomial is of degree $\leq D$ and $F_{\underline{t}}(x)=0$ for every $x\in S(\underline{t})$. \smallskip

Furthermore, an argument using Newton polygons shows that for our choice of the set $\mathcal{E}$ of monomials, the polynomials $F$ and $F_{\underline{t}}$ are relatively prime. We refer to \cite[\S1.3.4]{walarticle} where this argument is detailed. \medskip

Proof of $1.$ If $L< E$  the result is clear.
Suppose that $L \geq E$ and consider a minor, say $\Delta$, of order $E$. Up to permuting the lines and columns, one may assume that $\Delta =\det[(\underline{x_i}^{\underline{e}})_{1\leq i \leq E,~ \underline{e}\in \mathcal{E}}]$, or more specifically 
$$ \Delta =\begin{vmatrix}
   \underline{x_1}^{\underline{e_1}} &  \cdots & \underline{x_1}^{\underline{e_E}}\\
   \underline{x_2}^{\underline{e_1}} &  \cdots & \underline{x_2}^{\underline{e_E}}\\
   \vdots & \vdots & \\
   \underline{x_E}^{\underline{e_1}}  & \cdots & \underline{x_E}^{\underline{e_E}}
\end{vmatrix} $$
We will show that $\Delta =0$. To do this, we will show that the norm of $\Delta$ is divisible by a big power $p^\nu$ of $p$ and the height of $\Delta$ is bounded by a number $A$ such that $A^{\dd}<p^\nu$ and use the inequality $N(a) \leq H(a)^{\dd}$ for every $a\in O_K$.
\medskip 

For each $i=1, \cdots,E$, the pair $\underline{x_i}=(x_{i1},x_{i2})$ is in $S(\underline{t})$, in particular $\underline{x_i} \equiv \underline{t} \mod \ip$. Furthermore, we have assumed that $\displaystyle \frac{\partial F}{\partial X_1}(\underline{t})\notin \ip$. So we have 
$$ \frac{\partial F}{\partial X_1}(\underline{x_i}) \notin \ip \text{ and } F(\underline{x_i})=0 ~~ (i=1,\cdots, E).$$

We apply Lemma \ref{LFI} with $F$, $\underline{u} \in \widetilde{O_K}^2$ taken to be a lift of $\underline{t}$, and with $\underline{x}$ taken to be $\underline{x_i}$ ($i=1,\cdots,E$). Conclude that with $f_m(Y)$ the polynomials from lemma \ref{LFI}, we have $x_{i1}=f_m(x_{i2}) \mod \ip^m$ (for every $m\geq 1$ and $i=1,\cdots,E$). \smallskip

Set $\underline{w_i}=(w_{i1},w_{i2})=(f_m(x_{i2}),x_{i2})$, consider the matrix $M_0=(\underline{w_i}^{\underline{e}})_{1\leq i \leq E,~ \underline{e}\in \mathcal{E}}$ and set $\Delta_0=\det(M_0)$. For every $m\geq 1$, we have
$$\Delta \equiv \Delta_0 \mod {\ip^m}.$$

Because of the definition of $S(\underline{t})$, $x_{i2} \equiv t_2 \mod \ip$. Thus $x_{i2}$ can be written as $x_{i2}=t_2+y_{i2}$ where $t_2$ is independent of $i$ and $y_{i2} \in \ip$ for all $i=1,\cdots,E$.

\indent For $\underline{e}\in \mathcal{E}$, we then have $$\underline{w_i}^{\underline{e}}=f_m(t_2+y_{i2})^{e_1}(t_2+y_{i2})^{e_2}=g_{\underline{e}}(y_{i2})$$
for some polynomial $g_{\underline{e}}(Y) \in \widetilde{O_K}[Y]$.
\medskip

Next, we study the divisibility by $p$ of the norm of $\Delta_0$. Every column of $M_0$ corresponds to a polynomial $g_{\underline{e}}(Y)$. We claim that we can make some linear operations on the columns, without changing the determinant of $M_0$, in such a way to organize the columns by strictly growing degree. If $a$ is the smallest degree of some monomial, in first column, the degree $a$ monomial can be removed in every other column; iterating this process proves the claim. 
\medskip

\indent In the end, column $l$ has only elements in $\ip^{l-1}$ because it consists of polynomials in $y_{i2}$ where the first term is of degree $\geq l-1$ and $y_{i2} \in \ip$. Thus, the norm $N_{K/\Q}(\Delta_0)$ is divisible by $p^{ E(E-1)/2}$.  By choosing $m \geq E(E-1)/2$, we obtain that $N_{K/\Q}(\Delta)$ is divisible by $p^{E(E-1)/2}$.
\medskip

\indent Next, we estimate the height $H(\Delta)$. We have $H(x_{ij})\leq B$, $i=1,\cdots,E$, $j=1,2$. Denote by $S_E$ the permutation group of $E$ elements and for $\sigma \in S_E$, $\varepsilon(\sigma)$ the signature of $\sigma$. We have 
$$\Delta = \displaystyle \sum_{\sigma \in S_E} \varepsilon(\sigma) \prod_{i=1}^E\underline{x_{\sigma_i}}^{\underline{e_i}}.$$

For $v$ an archimedean place, 
\begin{align*}
|\Delta|_v & \leq E! \max_{1\leq j \leq E}(|x_{j1}|_v^{e_{11}} |x_{j2}|_v^{e_{12}}) \times \cdots \times \max_{1\leq j \leq E} (|x_{j1}|_v^{e_{E1}} |x_{j2}|_v^{e_{E2}}) \\
 & \leq E! \hskip 2pt B^{e_{11}+e_{12}} \times \cdots \times B^{e_{E1}+e_{E2}}
\end{align*}

We obtain :
\begin{align*}
H(\Delta) & = \max_{v/\infty} |\Delta|_v \\
& \leq E! \hskip 2pt B^{e_{11}+e_{12}} \times \cdots \times B^{e_{E1}+e_{E2}}
\end{align*}

To summarize, $N_{K/\Q}(\Delta)$ is divisible by  $p^{E(E-1)/2}$ but $|N_{K/\Q}(\Delta)| \leq  H(\Delta)^{\dd} \leq (E^EB^{E'})^{\dd}$.
We have then proved that under the condition $p^{E(E-1)/2}>(E^EB^{E'})^{\dd}$, we have $\Delta =0$. As $p\geq P$ and $\Delta$ is an arbitrary minor of $M$, we conclude that, under the assumption of the statement, the rank of $M$ is $\leq E-1$.
\end{proof}

\subsubsection{Technical conclusion of theorem \ref{HBE}.}

In order to apply proposition \ref{p1}, $P$ should satisfy the condition $$P>(E^{M_1}B^{M_2})^{\dd}.$$

\noindent where $M_1:=\displaystyle \frac{2}{(E-1)}$ and $M_2:=\displaystyle \frac{2E'}{E(E-1)}$. We refer to \cite{walarticle} for the following estimates of $M_1$ and $M_2$ :

 $$M_1 \leq \frac{2}{dD}+\frac{2}{D^2} ~~\text{and} ~~ M_2\leq \frac{1}{d}+\frac{6}{D}.$$

\noindent This leads to this condition $$P>(E^{2(dD)^{-1}+2D^{-2}}B^{d^{-1}+6D^{-1}})^{\dd}.$$
Note that $E\leq 2dD$ and, by an elementary study of function, we have $(2dD)^{2(dD)^{-1}+2D^{-2}} \leq e^8$. \\
We can choose
$$P=1+[(e^8B^{d^{-1}+6D^{-1}})^{\dd}] \leq (2^{13}B^{d^{-1}+6D^{-1}})^{\dd},$$
and take $B$ large enough so that $P$ satisfies the required condition of proposition \ref{p1}.

%%%%%%%%%%%%%%%%%%%%%%%%%%%%%%%%%%%%%%%%%%%%HBE%%%%%%%%%%%%%%%%%%%%%%%%%%%%%%%%%%%%%%%%%%%%%ù
%%%%%%%%%%%%%%%%%%%%%%%%%%%%%%%%%%%%%%%%%%%%%%%%%%%%%%%%%%%%%%%%%%%%%%%%%%%%%%%%%%%%%%%%%%%%%
To finish the proof of theorem \ref{HBE}, it remains to estimate the number $k$.
\smallskip

From lemma \ref{lemmeSFBP} and proposition \ref{p1},  $k=1+k'=1+\displaystyle \sum_{i=1}^r k''_{\ip_i}$ where $k''_{\ip}$ is the number of sets of type $S(\underline{t})$ in $S(F,B,\ip)$. Using the Lang-Weil bound \cite{frja}, we obtain $$k''_{\ip} \leq d(p+1+(d-1)(d-2)\sqrt{p})\leq 2d^3p.$$

Conjoining this with the upper bound for $p=N_{K/\Q}(\ip)$ from lemma \ref{lemmeSFBP}, we obtain :
 $$k\leq k_1 d^3 h(B)^2\frac{P}{\ln P} \ln (\frac{P}{\ln P}).$$ where $k_1$ is a constant depending on $K$.
 \medskip

As $\log P \geq \log(P/\log P)$, we have,
\begin{align*}
k & \leq k_1 d^3 h(B)^2 P\\
 & \leq c_2 d^3 \log^2(d^3H^+(F)B^{d-1}) (B^{d^{-1}+6D^{-1}})^\dd
\end{align*}
 where $c_2$ is a constant depending on $K$.

%%%%%%%%%%%%%%%%%%%%%%%%%%%%%%%%%%%%%%%%%%%%%%%%%%%%%%%%%%%%%%%%%%%%%%%%%%%%%%%
%%%%%%%%%%%%%%%%%%%%%%%%%%%%%%%%%%%%%%%%%%%%%%%%%%%%%%%%%%%%%%%%%%%%%%%%%%%%%%
%%%%%%%%%%%%%%%%%%%%%%%%%%%%%%%%%%%%%%%%%%%%%%%%%%%%%%%%%%%%%%%%%%%%%%%%%%%%%%%%
\subsection{Proof of theorem C.}

Let $F(X_1,X_2)$ monic in $x_2$ and $B$ as in theorem C. We keep the notation of \S 3.1.1.

\subsubsection{Non absolutely irreducible case.}
If $F$ is not absolutely irreducible, the following statement directly provides a bound for $N(F,B)$.
\begin{prop}
Let $F(X_1,X_2) \in O_K[X_1,X_2]$ of degree $d$, irreducible in $K[X_1,X_2]$ and not absolutely irreducible. Then $N(F,B)\leq 4d^4.$
\end {prop}
\begin{proof}
We will count the number of $x_1 \in O_K$ such that there exists $x_2\in O_K$ with $F(x_1,x_2)=0$. The same argument for $x_2$ will allow us to conclude.
\smallskip

\indent Let $(x_1,x_2)$ be a zero of $F$. Consider the factorization of $F$ in $\overline{K}[X_1,X_2]$. If $F(x_1,x_2)=0$ for $(x_1,x_2)\in O_K^2$ then $\varphi(x_1,x_2)=0$ for some irreducible factor $\varphi$ in $\overline{K}[X_1,X_2]$, also monic in $X_2$ and of degree $<d$ (as $F$ is not irreducible in $\overline{K}[X_1,X_2]$). We deduce that $\psi (x_1,x_2)=0$ for a $K$-conjugate $\psi$ of $\varphi$ over $K$ distinct from $\varphi$ (or else $F(X_1,X_2)$ would not be irreducible in $K[X_1,X_2]$). Furthermore, $\varphi$ and $\psi$ are not associated : if they were, as they are monic in $X_2$, they would be equal.
\smallskip

\indent Conclude that the product $\varphi \psi$ divide $F$. Thus $x_1$ is a double root of the polynomial $F(X_1,x_2)$. The number of such $x_1$ is bounded by the number of roots of the polynomial ${\rm disc}_{x_2}(F)$ which is of degree $\leq (2d-1)d \leq 2d^2$.
\smallskip

\indent With the same argument for $x_2$, we can say that the total number of points $(x_1,x_2) \in O_K^2$ such that $F(x_1,x_2)=0$ is at most $2d^2.2d^2=4d^4$. Hence $N(F,B)\leq 4d^4$.
\end{proof}

\subsubsection{The absolutely irreducible case.}
We assume $F(X_1,X_2)$ is irreducible in $\overline{K}[X_1,X_2]$ and is monic in $X_2$. For our applications, we need a bound for $N(F,B)$ which does not depend on the height $H(F)$ of $F$. We will use the following Siegel lemma for which we refer to \cite{trannumbers}.

\begin{lemme}\label{siegel}[Siegel lemma]
Let $K$ be a number field and $N$, $M$ two integers such that $1\leq M<N$. Let $H_0$ be a positive number and $a_{ij}\in K$, $1\leq i \leq N$, $1\leq j \leq M$, some algebraic numbers, not all zero, with height at most $H_0$. Then there exists a vector $\underline{x}\in O_K^N\backslash \{0\}$ such that : 
$$\displaystyle \sum _{i=1}^Na_{ij}x_i=0 \text{ , } j=1,...,M$$ and with $\displaystyle \max_{1\leq i\leq N}H(x_i)\leq C (C N H_0)^{M/(N-M)}$, where $C$ is a constant depending only on $K$.
\end{lemme} 

The constant $C$ that appears below is the constant that appears in lemma \ref{siegel}. 

\begin{prop}\label{indepH}
Let $F(X_1,X_2) \in O_K[X_1,X_2]$ be an irreducible polynomial in $\overline{K}[X_1,X_2]$ of degree $d$ and monic in $X_2$. Then
$$ N(F,B)\leq d^2+3 \text{ or } H(F) \leq C^{5d^2} 2^{8 d^{2}}d^{8 d^2}B^{4 d^3}.$$
\end{prop}

\begin{proof}
\indent Assume that $N(F,B)>d^2+3$. Set $R=d^2+4$, $N=(d+1)(d+2)/2$ and let $\underline{x_1},...,\underline{x_R}$ be $R$ zeroes of $F$ such that $H(x_{ij}) \leq B~(i=1,\cdots,R, ~~ j=1,2).$
\smallskip

\indent The total number of monomials of degree $\leq d$ in the indeterminates $X_{1}, X_{2}$ is $N$. %(denote them by $f_1,\cdots,f_N$). 
Let $A=(a_{i,j})$ be the $R \times N$ matrix of which the $i$-th line is composed of these $N$ monomials evaluated at $x_{i1},x_{i2}$ $i=1,\cdots,R$. The one column matrix $f \in O_K^N$, consisting of the coefficients of $F$ is a non trivial solution of the system $AX=0$.
\medskip

As $Af=0$, the rank of $A$, say $M$, is $<N$. Up to re-numbering the lines, we may assume that the system $AX=0$ is equivalent to its $M$ first lines.

\indent It follows from Lemma \ref{siegel} that the system has a non-zero solution $g \in O_K^N $ satisfying 
$$\max _{k=1,...,N} H(g_k) \leq C(CNB^{d})^{M/(N-M)}$$
(note that $H(a_{i,j})$ is bounded by $B^{d}$, $1\leq i \leq R$, $1\leq j\leq N$).
\smallskip

\indent Let $G(X_1,X_2)$ be the polynomial whose coefficients are the elements of $g$. $G$ is a non-zero polynomial of degree $\leq d$, its coefficients are in $O_K$, and it satisfies $G(x_{i1}, x_{i2})=0$ $(i=1,\cdots,R)$.
\medskip

By construction, the polynomials $F$ and $G$ have at least $d^2+4$ zeroes in common and are both of degree $\leq d$. By the Bezout theorem, these two polynomials are not relatively prime in $\overline{K}[X_1,X_2]$. As $F$ is irreducible and of degree $d$, we have $G=aF$ for some $a\in K$. Furthermore, as $F$ is monic in $X_2$, then $a\in O_K$ and $H(F) \leq H(G)$. Thus we have 
$$H(F)\leq H(G) \leq \max_{1\leq i\leq n} H(g_i) \leq C(CNB^{d})^{M/(N-M)} \leq C(CNB^{d})^{N}.$$
Note finally that $N\leq 4d^2$. Hence
$$H(F) \leq C^{5d^2} 2^{8 d^{2}}d^{8 d^2}B^{4 d^3}.$$
\end{proof}

%%%%%%%%%%%%%%%%%%%%%%%%%%%%%%%%%%%%%%%%%%%%%%%%%%%%%%%%%%%%%%%%%%%%%%%%%%%%%%%%%%%%%%%%%%%%%%%
%%%%%%%%%%%%%%%%%%%%%%%%%%%%%%%%%%%%majoration N(F,B)%%%%%%%%%%%%%%%%%%%%%%%%%%%%%%%%%%
%%%%%%%%%%%%%%%%%%%%%%%%%%%%%%%%%%%%%%%%%%%%%%%%%%%%%%%%%%%%%%%%%%%%%%%%%%%%%%%%%%%%%%%%%%%%
We can now finish the proof of theorem C. We deduce from Theorem \ref{HBE}, combined with the Bezout theorem that
$$N(F,B) \leq \sum_{i=1}^k \deg(F) \deg(F_i) \leq kdD \leq c_2 d^4 D \log^2(d^3H^+(F)B^{d-1}) (B^{d^{-1}+6D^{-1}})^\dd.$$

We recall that $D$ has to be chosen $\geq d$. We take $D=[d\log(B)+1]$. We have $B^{6(d\log(B))^{-1}} \leq 2^9$. We obtain :
$$N(F,B) \leq k_1 d^5 \log^2(d^3H^+(F)B^{d-1}) B^{\dd/d} \log(B)$$ where $k_1$ depends on $K$.
\smallskip

The bound $H(F) \leq C^{5d^2} 2^{8 d^{2}}d^{8 d^2}B^{4 d^3}$ from proposition \ref{indepH} gives :
$$N(F,B) \leq c_3 d^5 D \log^2(d^3 C^{5d^2} 2^{8 d^{2}}d^{8 d^2}B^{4 d^3} B^{d-1}) (B^{\dd/d} \log(B)).$$

Finally we obtain :
$$N(F,B) \leq c_5 d^{8} (\log B)^3 B^{\dd /d}.$$

\subsection{Proof of corollary C.}

\indent We work now with a polynomial $F(T,Y) \in O_K[T,Y]$ monic in $Y$ and  irreducible in $K[T,Y]$. We will estimate the number $N_{T}(F,B)$ of $t\in O_K$ such that $H(t) \leq B$ and the specialized polynomial $F(t,Y)$ has a root $y$ in $K$ (or, equivalently, in $O_K$ as $F(T,Y)$ is monic in $Y$). We recall that $m,n$ and $d$ are respectively the degree in $T$, $Y$ and the total degree of $F$.
\smallskip 

The following lemma based on the Liouville inequality, shows how to bound $H(y)$.

\begin{lemme}\label{liouv}
For all $t\in O_K$ but at most $m$ of them, the height of any $y\in O_K$ such that $F(t,y)=0$ is bounded by $ 2 (m+1) H(F)H(t)^m$.
\end{lemme}

\begin{proof}
We will use the Liouville inequality given in this form : if $P \in K[X]$ and $x\in \overline{K}$ such that $P(x)=0$, then $H(x) \leq 2H(P)$. \smallskip

 Write $F(T,Y)=a_0(T)Y^n+\cdots+a_n(T)$. Clearly we have $$\deg(a_i(T))\leq m \text{ and } H(a_i(T))\leq H(F),~ i=1,\cdots,n.$$ 
For $t\in O_K$ such that $a_0(t)\neq 0$ (the number of such $t$ is $\leq m$), the height of every solution $y\in O_K$ of the equation $F(t,Y)=0$ satisfies :
$$ H(y) \leq 2H(F(t,Y)).$$

We have $H(F(t,Y)) = H(a_0(t), \cdots , a_n(t)) =  \displaystyle \max_{1\leq i \leq n} \max _{v \in M_K} |a_i(t)|_v$ and
\begin{itemize}
\item for an archimedean place $v$, 
\begin{align*}
|a_i(t)|_v & \leq (m+1) \max_{1\leq i \leq n}(|a_{ij}|_v) \max (1,|t|_v)^m \\
& \leq (m+1) H_v(F) \max (1,|t|_v)^m,
 \end{align*}
\item for a finite place $v$, $|a_i(t)|_v \leq H_v(F) \max (1,|t|_v)^m.$
\end{itemize}
This yields
\begin{align*} \displaystyle \max_{1\leq i \leq n} \max_{v\in M_K} |a_i|_v & \leq (m+1) \max_{v\in M_K} H_v(F) \max_{v \in M_K} \max(1,|t|_v)^{m} \\
& \leq (m+1) H(F)H(t)^m
\end{align*}
This concludes the proof.
\end{proof}

Lemma \ref{liouv} gives $N_T(F,B) \leq N(F,B')$ with $B'=2(m+1)HB^m$. However, in order to obtain the right conclusion, we will apply this inequality, not to $F$, but to some polynomial $G$ deduced from $F$ by some change of variables. More precisely, we proceed as follows.

%%%%%%%%%%%%%%%%%%%%%%%%%%%%%%%%%%%%%%%%%%%%%%%%%%%%%%%%%%%%%%%%%%%%%%%%%%%%%%%%%%%%%%%%%%%%%%%%%%%%%%%%%%%%%%%%%%%%%%%%%%%%%%%%%%%%%%%%%%%%%%%%%%%%%%%%%%%%%%%%%%%%%%%%%%%%
%%%%%%%%%%%%%%%%%%%%%%%%%%%%%%%%%%%%%%%%%%%%%%%%%%%%%%%%%%%%%%%%%%%%%%%%%%%%%%%%%%%%%%%%%%%%%%%%
\begin{proof} [Proof of Corollary C]
We work with the $t\in O_K$ such that lemma \ref{liouv} can be applied. Adding the number of exceptional $t$ will only change the constant that appears in the final bound.
\smallskip

Let $H=\max(e^e, H(F))$, $L_1:=\log(H)$ and $L_2:=\log(\log(H))$. We have $L_2 \geq 1$. As $F(T,Y)$ is monic in $Y$, we have $d\leq n+m-1$. We may and will assume that $m\geq 1$ and $n\geq 1$. In particular $d\leq mn < mnL_1/L_2$.
\smallskip

Consider the polynomial 
$$G(T,Y)=F(T,T^E+Y)$$
 where $E=[mn\frac{L_1}{L_2}]+1 \leq 2mnL_1$. This polynomial is of degree $d'\in[nE,nE+m]$    and every zero $(t,y)\in O_K^2$ of $F$ corresponds to a zero  $(t,y')=(t,y-t^E)\in O_K^2$ of $G$.
Using the inequality $$H(a+b)\leq H(a) + H(b),$$ for every zero $(t,y')$ of $G$ such that $H(t)\leq B$, we have $$H(y') \leq (m+1)HB^m + B^E \leq  2(m+1)HB^{E}.$$
Thus, defining $B''= 2(m+1)HB^{E}$, we have $$N_T(F,B) \leq N(F,B')\leq N(G,B'').$$
Now use theorem C with $G$ and $B''$:
\begin{align*}
N(G,B'') &  \leq c_5 d'^{8} \log^{3}(B'') (B'')^{\dd /d'} \\
&  \leq c_5 (nE+m)^{8} \log^{3}(2(m+1)HB^E) (2(m+1)HB^{E})^{\dd /nE} \\
& \leq c_5 (nE+m)^{8} \log^{3}(2(m+1)HB^E) (2(m+1)H)^{\dd/ nE} B^{\dd /n}
\end{align*}
We have $E\leq 2d^2 \log H$, and as $1/nE \leq L_2/L_1$, we have $H^{1/nE} \leq log(H)$. Thus

$$ N(G,B'') \leq c_5(3d^3\log H)^{8}(4d^3 \log H \hskip 2pt \log B)^{3} (2^{\dd}d^{\dd} \log^{\dd} H) B^{\dd/n}.$$ 

Finally, we obtain $$N_T(F,B)\leq c_6d^{33+\dd}(\log H)^{11+\dd} B^{\dd/n} \hskip 2pt (\log B)^{3+\dd}.$$

\end{proof}

\section{Proof of theorem \ref{SPECTHM} and theorem 2.2.} \label{sectionresultinter}

\subsection{Proof of theorem \ref{SPECTHM}} \label{sectionspecthm}

Return to the situation of \S 2.1 : a regular Galois extension $F/K(T)$ of group $G$ is given. 
\smallskip

Fix a good prime $\ip$ for $F/K(T)$ and an associated union $\frob_{\ip}$ of conjugacy classes of $G$. The following result generalizes \cite[proposition 5.1]{debconjmalle}, proved in the case $K=\Q$.

\begin{prop}\label{cosets}
The set $$\tau(\frobp):=\{t_0 \in O_K \mid ~ t_0 \notin \mathbf{t} \mod \ip, ~~ {\rm Frob}_{\ip} (F_{t_0}/K) \in \frobp\}$$ is a union of cosets modulo $\ip$ and the number $\nu(\frobp)$ of these cosets satisfies
\begin{align*}
\nu(\frobp) & \geq \frac{|\frobp|}{|G|}\times (q+1-2g\sqrt{q}-|G|(r+1)) \\
\nu(\frobp) & \leq \frac{|\frobp|}{|G|}\times (q+1+2g\sqrt{q})
\end{align*}
where $q=N_{K/\Q}(\ip)$.
\end{prop}
We omit the proof which merely consists in changing the prime number $p$ to the prime ideal $\ip$ in the proof of \cite[proposition 5.1]{debconjmalle}.
\smallskip

Consider the prime numbers $p_0$ and $p_{-1}$ given in \S 2. Let $x>0$ be a real number. Let $S_{[p_0,x]}$ be the set of all primes of $K$ over the interval $[p_0,x]$ and let $\frob_{[p_0,x]}$ be a Frobenius data on $S_{[p_0,x]}$. Next, with $S_{]p_{-1},p_0[}$ consisting of the prime ideals over the interval $]p_{-1}, p_0[$, set $S_x=S_{[p_0,x]} \cup S_{]p_{-1},p_0[}$. 
\smallskip

Consider the Frobenius data $\frob_x$ on $S_x$ obtained by adding to the Frobenius data $\frob_{[p_0,x]}$ some local conditions over the primes of $K$ over the interval $]p_{-1},p_0[$ in this manner : to every conjugacy class of $G$, we associate a prime $\ip \in S_{]p_{-1},p_0[}$ in such a way that every conjugacy class of $G$ appears in the Frobenius data $\frob_x$; for the other ideals in $S_{]p_{-1},p_0[}$, we take $\frob_{\ip}=G$ ($\frob_{\ip}$ can be chosen arbitrary). Set $I=\displaystyle \prod_{\ip\in S_x}\ip$.

\begin{lemme}\label{ideal}
We have $I=\displaystyle \prod_{1\leq i \leq n} (p_i O_K)$ and $I \cap \Z = (p_1\cdots p_n)\Z$. Denote by $\{p_1, \cdots, p_n\}$ the set of prime numbers in the interval $]p_{-1},x]$.
\end{lemme}
\begin{proof}
The set $S_x$ is the set of all prime ideals of $O_K$ over $p_1,\cdots,p_n$. Using that all primes $\ip \in S_x$ are unramified in $K$ from the definition of $p_{-1}$, we obtain
$$I=\displaystyle \prod_{\ip/p_1}\ip \cdots \prod_{\ip/p_n}\ip=(p_1O_K)\cdots (p_nO_K).$$
%=\prod_{1\leq i\leq n}(p_iO_K).$$
For $i\in\{1,\cdots,n\}$, we have $p_i O_K \cap \Z =p_i\Z$. 
The next argument shows that $$ (p_1 O_K) \cdots (p_n O_K) \cap \Z = (p_1 \cdots p_n)\Z .$$

\indent Inclusion $\supset$ is obvious : $p_1 \ldots p_n \in (p_1O_K) \ldots (p_nO_K) \cap \Z$. As $\Z$ is a P.I.D, the ideal $(p_1 O_K) \ldots (p_n O_K) \cap \Z$ is of the form $a \Z$ for some $a\in \Z$. From $(p_1 O_K) \ldots (p_n O_K) \cap \Z \subset p_i O_K \cap \Z$, we deduce that $p_i \mid a$, $i=1, \cdots, n$. As $p_1, \cdots ,  p_n$ are distinct, $p_1 \ldots p_n \mid a$ whence the desired inequality.
\end{proof}

Denote the intersection $\displaystyle \bigcap_{\ip \in S_x} \tau(\frobp)$ by $\tau(S_x,\frob_x)$. It follows from the Chinese remainder theorem that $\displaystyle \tau(S_x,\frob_x)$ contains $\mathcal{N}(S_x,\frob_x)=\displaystyle \prod_{\ip \in S_x} \nu(\frobp)$ cosets modulo $I$. The following proposition is a more precise and more technical form of theorem \ref{SPECTHM}. It involves the following notation.

\begin{itemize}
\item for a Frobenius data $\frob_S= (\frob_{\ip})_{\ip\in S}$, as in \S 2.1, the density of $\frob_S$,  denoted by $\chi(\frob_S)$, is the product of all  $|\frob_{\ip}|/|G|$ for $\ip \in S$,
\item for a positive real number $x$, the number $\pi(x)$ is the number of primes $\leq x$ and $\Pi(x)$ is the product of all prime numbers $p\leq x$. Recall that $\displaystyle \pi(x)\sim x/\log x \text{ and } \log (\Pi(x)) \sim x$ when $x\rightarrow +\infty$.
\item  for a set $S$ of prime ideals in $O_K$, the number $\Pi(S)$ is the product of all primes numbers $p$ such that $p=p_{\ip}$ for some prime ideal $\ip \in S$\footnote{Recall that $p_{\ip}$ denote the prime number such that $\ip \cap \Z=p\Z$.}.
\end{itemize}

\begin{prop}\label{representants}
1. If $t_0 \in O_K$ is any representative of one of the cosets modulo $I$ in $\tau(S_x,\frob_x)$ then ${\rm Gal}(F_{t_0}/K)=G$ and $t_0\in \tau(\frobp)$ for each $\ip \in S_x$. %, $F_{t_0}/K$ is unramified at each $\ip\in S_x$ and $\rm{Frob}_{\ip}(F_{t_0}/K)\in \frobp$ for every $\ip \in S_x$.
\smallskip

\noindent 2. If $x$ is suitably large, $$ \mathcal{N}(S_x,\frob_x) \geq \displaystyle \chi(\mathcal{F}_x) \times \frac{\Pi(x)^{\dd}}{(\Pi(p_{-1}))^{\dd}} \times (\frac{1}{2r|G|})^{\dd \pi(x)}.$$
3.Fix a $\Z$-basis $\underline{e}=(e_1,\cdots,e_n)$ of $O_K$ and denote by $H(\underline{e})$ the height of $\underline{e}$. For every coset modulo $I$ in $\tau(S_x,\frob_x)$, there exists a representative $t_0\in O_K$ of height $H(t_0)\leq \displaystyle \frac{\dd H(\underline{e})}{\Pi(p_{-1})} \Pi(x)$.

\end{prop}
\begin{proof}
1. From the definition of $\tau(S_x,\frob_x)$, we have that ${\rm Frob}_{\ip}(F_{t_0}/K)\in \frobp$ for every $\ip \in S_x$. From the Frobenius condition on the primes of $S_{]p_{-1},p_0[} \subset S_x$, the subgroup ${\rm Gal}(F_{t_0}/K)\subset G$ meets all conjugacy classes of $G$, so it is the whole group $G$ by a lemma of Jordan \cite{jordan}.
\medskip

2. Using proposition \ref{cosets}, we have, for $q=N(\ip)$ with $\ip \in S_x$.
\begin{align*}
 \mathcal{N}(S_x,\frob_x) & = \prod_{\ip \in S_x} \nu(\frobp) \\
 & \geq \prod_{\ip \in S_x} \frac{|\frobp|}{|G|}\times (q+1-2g\sqrt{q}-|G|(r+1)) \\
 & \geq \chi(\frob_x) \times \prod_{\ip \in S_x} q \times \prod_{\ip \in S_x}(1+\frac{1}{q}-\frac{2g}{\sqrt{q}}-\frac{(r+1)|G|}{q})
\end{align*}

\noindent As in \cite{debconjmalle}, using that $g<r|G|/2-1$ (if $|G|>1$; from the Riemann-Hurwitz formula) and that $q\geq r^2 |G|^2$ for each $\ip \in S_x$ (from the choice of $p_{-1}$), we have $$\displaystyle 1+\frac{1}{q}-\frac{2g}{\sqrt{q}}-\frac{(r+1)|G|}{q} \geq \frac{1}{2r|G|}.$$

As all primes in $S_x$ are unramified, we have $\displaystyle \prod_{\ip \in S_x} N(\ip)=\Pi(S_x)^{\dd}= \displaystyle \left(\frac{\Pi (x)}{\Pi( p_{-1})}\right)^{\dd}$ and \\ ${\rm card} (S_x) \leq \dd \pi(x)$. Hence, we obtain
$$\mathcal{N}(S_x,\frob_x) \geq \displaystyle \chi(\mathcal{F}_x) \times \frac{(\Pi(x))^{\dd}}{(\Pi(p_{-1}))^{\dd}} \times \left(\frac{1}{2r|G|}\right)^{\dd \pi(x)}.$$

\medskip

3. We have $O_K=\{\displaystyle \sum_{i=1}^{\dd} m_i.e_i ~~\mid ~~ m_i \in \Z ~~  i =1,\cdots,\dd \}$ and so $$O_K/I =\{\sum_{i=1}^{\dd}\overline{m_i}.\overline{e_i} \mid \overline{m_i} \in \Z /\Z \cap I ~~ i= 1,\cdots,\dd \}.$$

From lemma \ref{ideal}, $\Z/\Z\cap I= \Z/\Pi(S_x)\Z$. 
Every coset modulo $I$ in $\tau(S_x,\frob_x)$  has a representative $t=\displaystyle \sum_{i=1}^{\dd} m_i.e_i$ in $O_K$ such that $1\leq m_i \leq \Pi(S_x)$, $i=1,\cdots,\dd$. Next we have 
\begin{itemize}
\item for each finite place $v$, $|t|_v= \displaystyle \left|\sum_{i=1}^{\dd} m_i.e_i\right| \leq \max_{1\leq i\leq \dd}(|e_1|_v,\cdots,|e_{\dd}|_v)$.
\item for each archimedean place $v$, $|t|_v =\displaystyle \left|\sum_{i=1}^{\dd} m_i.e_i\right| \leq \dd \Pi(S_x) \max_{1\leq i\leq \dd}(|e_1|_v,\cdots,|e_{\dd}|_v)$
\end{itemize}
Whence $H(t)\leq \dd \Pi(S_x) H(e_1,\cdots,e_{\dd})=\displaystyle \frac{\dd H(\underline{e})}{\Pi(p_{-1})}\Pi(x).$
\end{proof}

\textbf{Proof of theorem \ref{SPECTHM}.}
For a positive number $B$, we let $x=p_B$ be the biggest prime number such that $\Pi(p_B).p_B \leq B$. Denote by $q_B$ the next prime number. As $\Pi(q_B)=\Pi(p_B).q_B$, we have 

$$p_B \Pi(p_B) \leq B < q_B^2 \Pi(p_B)\leq 4p_B^2 \Pi(p_B);$$

\noindent the last inequality uses the classical estimate $q_B\leq 2p_B$.\smallskip

Taking the log of these terms yields $$\frac{\log(\Pi(p_B))}{p_B} + \frac{\log p_B}{p_B} \leq \frac{\log B}{p_B} \leq \frac{\log (\Pi(p_B))}{p_B}+ \frac{2 \log 2p_B}{p_B}$$
which shows that $$p_B \sim \log B ~~\text{ when }B \rightarrow \infty.$$

Take a number $B$ which satisfies the following conditions :
\begin{itemize}
\item $\frac{\log B}{2}\leq p_B \leq 2\log B$,
\item $p_B \geq \displaystyle \frac{\dd H(\underline{e})}{H(p_{-1})}$,
\item $\pi(p_B) \leq 2 \log B /\log \log B$.
\end{itemize}
It suffices to take $B$ suitably large depending on $K$, $H(\underline{e})$, $H(p_{-1})$. \smallskip

As in theorem \ref{SPECTHM}, let $S_B$ be the set of primes of $K$ over the interval $[p_0, \log B /2]$. The interval $[p_0, \log B/2]$ is contained in the interval $[p_0, p_B]$ of proposition \ref{representants}. Let $\mathcal{F}_B$ be a Frobenius data on $S_B$. Extend it to a Frobenius data $\mathcal{F}_x=\mathcal{F}_{p_B}$ on the set $S_x=S_{p_B}$ of primes over the interval $[p_{-1}, p_B]$ (by defining $\frobp$ arbitrarily for every $\ip$ over some prime in $[\log B/2,p_B]$).
\medskip
 
Next, use proposition \ref{representants} with $x=p_B$, the set $S_{p_B}$ and the Frobenius data $\mathcal{F}_{p_B}$. Note that the upper bound for $H(t_0)$ in proposition \ref{representants} (3) is $\leq p_B \Pi(p_B)$ and so $\leq B$. Conclude from proposition \ref{representants} (2) that the number $\mathcal{N}$ of $t_0\in O_K$ such that ${\rm Gal}(F_{t_0}/K)=G$, $H(t_0) \leq B$ and for all $\ip \in S_B$, $\rm{Frob}_{\ip}(F_{t_0}/K) \in \frobp$ satisfies 
$$\mathcal{N}\geq \displaystyle \chi(\mathcal{F}_{p_B}) \times \frac{\Pi({p_B})^{\dd}}{(\Pi(p_{-1}))^{\dd}} \times (\frac{1}{2r|G|})^{\dd \pi(p_B)}.$$

\noindent Furthermore, we have $$\chi(\mathcal{F}_{p_B})=\displaystyle \prod_{\ip \in S_{p_B}} \frac{|\frobp|}{|G|} \geq \frac{1}{|G|^{|S_{p_B}|}} \geq \frac{1}{|G|^{\dd \pi(p_B)}}.$$ 
and
\begin{itemize}
\item $\displaystyle \Pi(p_B)^{\dd} = \frac{(4 p_B^2 \Pi(p_B))^{\dd}}{(2p_B)^{2 \dd}} \geq \frac{B^{\dd}}{(2p_B)^{2 \dd}}$,
\item $(2p_B)^{2 \dd} \leq c_1^{\log B /\log \log B}$ for a constant $c_1$ depending on $F/K(T)$.
\end{itemize}

Finally, using that $\pi(p_B) \leq 2 \dd \log B / \log \log B$, we obtain
\begin{align*}
\mathcal{N}\geq \frac{B^{\dd}}{c^{\log B / \log \log B}}
\end{align*}
for a constant $c$ depending on $F/K(T)$.
\qed

\begin{theorem} \label{selftwistedcover}
Let $F/K(T)$ be a regular Galois extension of group $G$. There exists an integer $N \leq |Aut(G)|$ and some polynomials $\widetilde{P}_1,\cdots,\widetilde{P}_N \in O_K[U,T,Y]$, irreducible in $\overline{K(U)}(T)[Y]$, of degree $\deg_Y(\widetilde{P}_i)=|G|$, monic in $Y$ and a finite set $\varepsilon \subset K$ such that the following holds : 
\begin{itemize}
\item all the affine curves $\widetilde{P_i}(U,t,y)=0$ (over $\overline{K(U)}$) are of the same genus $g_F$,
\item for every $u_0 \in O_K\backslash \varepsilon$, $\widetilde{P_i}(u_0,T,Y)$ is irreducible in $\overline{K}(T)[Y]$ and the affine curve $\widetilde{P_i}(u_0,t,y)$ is of genus $g_F$, $i=1,\cdots,N$,
\item for every $t_0 \in O_K$ which is not a branch point of $F/K(T)$, 
$$F_{t_0}/K = F_{u_0}/K \iff \exists i \in\{1,\cdots,N\}, \exists y_0 \in K : \widetilde{P}_i(u_0,t_0,y_0)=0.$$
\end{itemize}
\end{theorem}

This result is the special case of \cite[theorem 2.16]{noparametric} for which $F/K(T)= L/K(T)$. Each polynomial $\widetilde{P}_i$ is an affine model of the $K(U)$-regular cover $\widetilde{f}_i : \widetilde{X}_i \rightarrow \PP_{K(U)}$ that appears there and is obtained somehow by twisting $f_i$ by itself ($i=1,\cdots,N$). Except for a finite number of them, the $K$-points on $\widetilde{X}_i\vert_{u_0}$ that appear in \cite{noparametric} correspond to the zeroes $(t_0,y_0)$ of the polynomial $\widetilde{P}_i(u_0,T,Y)$, $i=1,\cdots N$.
\medskip

\textbf{Diophantine estimates.}
The constants $c_i$ below, $i=1,2,3$ depend only on the extension $F/K(T)$. We have :
\begin{itemize}
\item $\deg(\widetilde{P}(u_0,T,Y)) \leq \deg(\widetilde{P})=c_1$
\item $\deg_Y(\widetilde{P}(u_0,T,Y)) = \deg_Y(\widetilde{P})=|G|$
\item $H(\widetilde{P}(u_0,T,Y)) \leq c_2 H(u_0)^{c_3} \leq c_2 B^{c_3}$ \\
\end{itemize}

For real numbers $g,D,H,B\geq 0$ and $d_Y \geq 2$, consider all polynomials $Q\in O_K[T,Y]$ monic in $Y$ and irreducible in $\overline{K}(T)[Y]$ such that 
\begin{itemize}
\item $\deg_Y(Q)=d_Y$
\item $\deg(Q) \leq D$
\item $H(Q) \leq H$
\item the curve $Q(t,y)=0$ is of genus $\leq g$.
\end{itemize}

\noindent For each such polynomial $Q$, the number of $t \in O_K$ of height $H(t) \leq B$ and such that $Q(t,y)=0$ for some $y\in O_K$ is finite. Denote by $Z(g,D,d_Y,H,B)$ the maximal cardinality of all these finite sets when $Q$ ranges over all polynomials satisfying the above conditions.
\smallskip

As in theorem \ref{HilbertMalle}, let $B$ be a positive number and $\mathcal{H} \subset O_K$ be a subset consisting of $t_0$ such that ${\rm Gal}(F_{t_0}/K)=G$ and $H(t_0) \leq B$. From theorem \ref{selftwistedcover}, for every $u_0\in \mathcal{H}$, the number of $t_0\in \mathcal{H}$ such that $F_{t_0}/K=F_{u_0}/K$ is $\leq N \hskip 2pt Z(g_F,c_1,|G|,c_2B^{c_3},B)$. Let $E$ be the cardinality of $\varepsilon$ of theorem \ref{selftwistedcover}, we obtain 
$$\mathcal{N}(B,\mathcal{H}) \geq \frac{|\mathcal{H}|-E}{N Z(g_F,c_1,|G|,c_2B^{c_3},B)}.$$

From corollary C, we have
$$Z(g_F,c_1,|G|,c_2B^{c_3},B) \leq c_5B^{\dd/|G|}(\log B)^{c_6}.$$
Finally we obtain
$$\mathcal{N}(\mathcal{H}) \geq \frac{|\mathcal{H}|-E}{B^{\dd/|G|}(\log B)^{\gamma}}.$$

\newpage
\selectbiblanguage{british}
\bibliographystyle{alpha}
\bibliography{bibliographie}
\thispagestyle{empty}

\end{document}